\documentclass[onefignum,onetabnum]{siamonline250211}

\usepackage{lipsum}
\usepackage{amsfonts}
\usepackage{graphicx}
\usepackage{epstopdf}
\usepackage{algorithmic}
\ifpdf
  \DeclareGraphicsExtensions{.eps,.pdf,.png,.jpg}
\else
  \DeclareGraphicsExtensions{.eps}
\fi
\usepackage{youngtab}
\usepackage{amsmath,amscd}
\usepackage{graphicx}
\usepackage{booktabs}
\usepackage{subcaption}
\usepackage{listings} 
\usepackage{amssymb}
\usepackage{mathtools}
\usepackage{algorithm}
\usepackage{xr-hyper}
\usepackage{hyperref}

\newcommand{\defeq}{\vcentcolon=}
\newsiamremark{remark}{Remark}
\newsiamremark{hypothesis}{Hypothesis}
\crefname{hypothesis}{Hypothesis}{Hypotheses}
\newsiamremark{exmp}{Example}
\newsiamremark{assumption}{Assumption}

\headers{UQ for First Order Methods}
{C. McMeel And P. Parpas}

\title{Uncertainty Quantification for the Gradient and Accelerated Gradient Descent Methods on Strongly Convex Functions
\thanks{Submitted to the editors \today.}
}

\makeatletter
\newcommand*{\addFileDependency}[1]{%
  \typeout{(#1)}%
  \@addtofilelist{#1}%
  \IfFileExists{#1}{}{\typeout{No file #1.}}%
}
\makeatother

\author{Conor McMeel\thanks{Department of Computing, Imperial College London 
  (\email{c.mcmeel18@alumni.imperial.ac.uk}).}
\and Panos Parpas\thanks{Department of Computing, Imperial College London 
  (\email{panos.parpas@imperial.ac.uk}).}}

\usepackage{amsopn}

\ifpdf
\hypersetup{
  pdftitle={Uncertainty Quantification for Gradient and Accelerated Gradient Descent Methods on Strongly Convex Functions},
  pdfauthor={C. McMeel, And P. Parpas}
}
\fi

\begin{document}

\maketitle

\begin{abstract}
We consider minimising a strongly convex function that depends on an uncertain parameter $\theta$, so that the optimum $x^*(\theta)$ is itself a function of $\theta$. We propose an efficient method to compute $x^*(\theta)$ and its statistics, using a chaos expansion of $x^*(\theta)$ in a truncated basis and first-order methods that compute the optimal coefficients.
We give non-asymptotic convergence rates for gradient descent and accelerated gradient descent when the truncation level, and hence the dimension of the optimisation problem, grows within a single run, and the objective is observed only through noisy gradients. Our analysis exploits convexity and does not rely on diminishing step sizes: with a constant step size, the iterates converge linearly to a neighbourhood of the solution whose radius is set by the truncation error and the Monte Carlo variance. When the gradient is exact at each $\theta$, the error then decays at the optimal rate of the truncation error, and we bound the number of gradient evaluations needed to reach a given accuracy in terms of the smoothness of the solution. The accelerated method improves the iteration complexity from order $\kappa$ to order $\sqrt{\kappa}$, and in our experiments needs fewer gradient evaluations than the state of the art, with a speedup that grows with the condition number. Growing the basis also reduces the variance of early gradient estimates compared with a large fixed basis, which permits larger step sizes.
\end{abstract}



\section{Introduction}
Consider the following optimisation problem,
\begin{equation}\label{eq:intro prob}
		x^* = \arg \min_{x \in L^2_\pi}\int_{\Theta} f(x(\theta), \theta)\,\pi(d\theta), \qquad f(x(\theta), \theta) = \int_{V} F(x(\theta), \theta,v) \nu(\theta, dv),
	\end{equation}
	where $\theta \in \Theta\subset\mathbb{R}^d$ is an uncertain parameter that is distributed according to a probability distribution $\pi$ and where $v$ represents an uncertain variable distributed according to $\nu$.  The probability measure $\nu$ may also depend on $\theta$.  
	The optimiser $x^*$ depends on $\theta$ and we seek to find a function $x^*(\theta)$ that belongs to $L^2_\pi$ (the space of all square integrable functions with respect to the probability density function $\pi$), and minimises the expected value of $F$ with respect to $v$. 
	We denote this expectation by $f(x(\theta), \theta)$. Therefore we are looking for a function $x^*(\theta)$ such that on average (in an $L^2_\pi$ sense), we have that $f(x^*(\theta), \theta)$ is a minimum of $f(x(\theta), \theta)$ . We make this setting clearer with the following concrete example:

    \begin{exmp} \label{exmp::marko}
	    In the \textit{Markowitz problem} \cite{markowitz1952portfolio}, an investor allocates capital $w$ across $n$ risky assets with forecast returns $r$ and covariance matrix $\Sigma$, by minimising $-w^T r + \frac{\lambda}{2} w^T \Sigma w$ subject to $w^Te = 1$, where $e$ is a vector of ones. This simple case is easily solvable \cite{rubinstein2002markowitz}. In applications, constraints are often imposed through quadratic penalties. Below, a cost on turnover from a current portfolio $w_0$ also makes the problem strongly convex, and a penalty replaces the budget constraint. The covariance matrix is notoriously difficult to estimate \cite{ledoit2022power}, and we represent this uncertainty by writing $\Sigma(\theta)$. In the form of Equation \eqref{eq:intro prob}, we then seek the function $w^*(\theta)$ that minimises $\int_{\Theta} f(w(\theta),\theta)\,\pi(d\theta)$ over $w \in L^2_\pi$, where
	    \begin{equation*}
	        f(w,\theta) = -w^T r + \frac{\lambda}{2} w^T \Sigma(\theta) w + \frac{\tau}{2} ||w - w_0||_2^2 + \frac{\eta}{2} \left(w^Te - 1\right)^2 .
	    \end{equation*}
	\end{exmp}
	
	For other concrete instances of the problem, we direct the reader to \cite{crepey2020uncertainty} and \cite{sullivan2015introduction}. Our main assumptions are that the function $f$ is strongly convex, we have noisy gradient information and that the model is to be optimised with a first order method (i.e. second order information is either unavailable or too expensive to compute). 
	Our precise assumptions and model description appear in Section \ref{section::AlgoIntro}. 
	
	In order to develop a numerical scheme to solve the problem above we choose a suitable basis of $L^2_\pi$, $B_i(\theta)$, which is indexed from $0$ to $\infty$, and write a (infinite-dimensional) basis decomposition for $x^*(\theta)$,
	\begin{equation}
	    x^*(\theta) = \sum_{i=0}^{\infty} u_i B_i(\theta).
	\end{equation}
	We propose a gradient descent method and its accelerated variant to compute the coefficients $u_i$. To see the benefit, consider the naive approach: sample $\theta_1, \ldots, \theta_N \sim \pi$, solve $\min_{y \in \mathbb{R}^n} f(y,\theta_j)$ for each sample by gradient descent, and use the $N$ solutions to infer $x^*(\theta)$ or its statistics. We compare methods by the total number of evaluations of $\nabla F$. If $F$ does not depend on $v$, each run converges linearly, so up to logarithmic factors its cost does not grow as the target accuracy $\delta$ shrinks. However, estimating even the mean of $x^*(\theta)$ to accuracy $\delta$ requires $N = \mathcal{O}(\delta^{-2})$ runs, whatever the regularity of $x^*$, and recovering the whole function requires more still: for example, estimating the variance $\sum_{i \geq 1} u_i^2$ of a scalar $x^*$ with a Jacobi polynomial basis requires $\varepsilon^{-(1+d/(2(\eta-1)))}$ runs, where $\eta$ is the order of differentiability of $x^*$ \cite{crepey2020uncertainty}.

Our method instead exploits the regularity of $x^*$. If keeping the first $m$ coefficients leaves an error of order $m^{-p}$, where $p$ grows with the smoothness of $x^*$, then $m = \mathcal{O}(\delta^{-1/p})$ coefficients suffice. We compute all of them in a single linearly convergent run, in which each iteration estimates the first $m$ coefficients of the gradient by Monte Carlo with a number of samples growing like $m^q$, where $q = 1$ for bounded bases such as the Fourier basis and $q = 2$ for the Legendre basis in one dimension (Theorem \ref{theorem::IntegralError}). Up to logarithmic factors and constants depending on the conditioning, the total cost is $\mathcal{O}(\delta^{-q/p})$ evaluations (Corollary \ref{cor:complexity}), an improvement on $\mathcal{O}(\delta^{-2})$ whenever $p > q/2$. It also recovers the whole function $x^*(\theta)$, and hence quantities such as the probability that an asset in Example \ref{exmp::marko} is held short. As with any spectral method, the advantage relies on $x^*$ being smooth relative to the dimension $d$ of $\theta$, whereas the Monte Carlo rate does not depend on $d$.

    \subsection{Our Contributions}
    We analyse gradient descent and accelerated gradient descent methods for the chaos coefficients of $x^*(\theta)$ in which the approximation space grows within a single run: iteration $k$ uses the first $m_k$ basis functions. Such methods have been analysed asymptotically in a stochastic approximation setting \cite{crepey2020uncertainty}, and non-asymptotically at a fixed truncation level \cite{dong2025beyond}. Our results are non-asymptotic and allow $m_k$ to grow.
    \begin{itemize}
        \item \emph{Changing approximation spaces within a single run.} We combine a bound on the distance between the level $m$ optimum and $u^*$ (Lemma \ref{lemma::LinConvRemainder}) with fixed-level analyses (Theorem \ref{theorem::GDFixedUQ}) and, for the accelerated method, a bound on the change in the Lyapunov function when basis functions are added (Lemma \ref{lemma::levelChangeLyap}). This gives linear convergence to a neighbourhood of $u^*$ for a growing basis (Theorems \ref{theorem::gradConvLinear} and \ref{theorem::AccGradOverall}), whose radius has a part that vanishes as the basis grows and a part set by the variance $V$ of $\nabla F$ in $v$. When $V = 0$, the error stays within a constant factor of the lower bound $||R_{m_k}||_2$ (Proposition \ref{prop::tracking}), and we bound the number of gradient evaluations of a single run that does not need the target accuracy in advance (Corollary \ref{cor:complexity}).
        \item \emph{Basis-dependent bounds on the gradient noise.} The Monte Carlo error of the gradient estimates depends on the basis through $Q_m$ (Theorem \ref{theorem::IntegralError}). This sets how the sample sizes must grow for the step size to stay constant, and means that a growing basis gives lower-variance early gradient estimates than a large fixed basis, permitting larger step sizes (Section \ref{sssec:varRed}).
        \item \emph{Acceleration with a noisy oracle.} We allow $f$ to be itself an expectation over a second random variable $v$, combine the resulting noise bounds with the robust accelerated method of \cite{aybat2019universally}, and extend it to a growing basis (Theorem \ref{theorem::AccGradOverall}).
    \end{itemize}
	In our experiments, the accelerated method needs fewer gradient evaluations than the adaptive method of \cite{dong2025beyond} when both use the same rule to grow the basis, by a factor that grows roughly like $\kappa^{0.4}$ with the condition number.

    \subsection{Previous Work}
	Much of the previous work on chaos expansions for uncertainty quantification with gradient methods has focused on Robbins--Monro stochastic approximation (SA) \cite{robbins1951stochastic}, which has become popular for large-scale problems such as those in machine learning \cite{bottou2018optimization}; see \cite{shapiro2021lectures, vapnik1999nature} for background. In \cite{kulkarni2009finite} a similar spectral approach is taken, but truncated to a fixed finite family $B_0, \ldots, B_m$ at all iterations, with an SA procedure for the coefficients followed by an error analysis of the finite-dimensional approximation. At the other extreme, \cite{yin1990h} give an SA algorithm defined in an infinite-dimensional Hilbert space, which is not practically implementable. We instead increase the number of basis functions over time, a standard device in numerical analysis known as a sieve method. Sieve SA procedures have been shown to satisfy asymptotic normality \cite{nixdorf1984invariance}, almost-sure convergence \cite{yin1992h}, and an almost-sure loglog rate \cite{chen2002asymptotic}. Most recently, \cite{crepey2020uncertainty} show asymptotic convergence of an SA sieve algorithm under only standard SA local separation and a local strong convexity type assumption, but with step size conditions tied to the number of basis functions and the decay of the truncation error.

	These references treat general root-finding problems with the decaying step sizes typical of SA, and give asymptotic results. We instead restrict to the minimisation of a strongly convex function with Lipschitz continuous gradient, and study gradient descent and the robust accelerated method of \cite{aybat2019universally}. This gives non-asymptotic rates with constant step sizes, which depend on $m_k$ only through $Q_{m_k}/M_k$ and so can be held constant by growing the sample sizes with $Q_{m_k}$ (Section \ref{sec::GD}).

    The closest work to ours is that of Dong, Grigas and Gupta \cite{dong2025beyond}, which appeared after an earlier version of the present manuscript \cite{mcmeel2021uq}. They represent the solution path of a family of smooth, strongly convex problems $\min_{x} h(x,\lambda)$ in a fixed basis, and solve the stochastic problem obtained by sampling $\lambda$ from a user-chosen distribution; in our framework, this is the problem at a fixed level $m$, with $\pi$ the sampling distribution. Using the relaxed weak growth condition of \cite{gower2019sgd}, they show that constant step-size stochastic gradient descent converges linearly to a neighbourhood whose radius is proportional to the smallest worst-case suboptimality achievable in the basis, and derive end-to-end complexity bounds for a Legendre basis on $[-1,1]$. The main difference is that we analyse a truncation level that grows within a single run, whereas their adaptive variant is presented as a heuristic; this is what gives Proposition \ref{prop::tracking} and Corollaries \ref{corol::exactLinear}, \ref{corol::exactLinearAGD} and \ref{cor:complexity}. We also allow $f$ to be an expectation over $v$, and extend acceleration to a growing basis, with a neighbourhood of order $\kappa\varepsilon$ compared with $\kappa^{3/2}\varepsilon$ for gradient descent. Finally, we bound the $L^2_\pi$ error, which is natural when the object of interest is the statistics of $x^*(\theta)$, whereas \cite{dong2025beyond} bound the worst-case suboptimality over $\lambda$, which suits presenting an entire solution path.
	
	\subsection{Outline}
    Section \ref{section::AlgoIntro} sets out the problem and notation, Section \ref{sec::Algo} the algorithm and two auxiliary results, Sections \ref{sec::GD} and \ref{sec::AGD} the convergence of gradient descent and accelerated gradient descent with a growing basis, and Section \ref{sec::Experiments} the experiments. Section \ref{sec::Conclusion} concludes.
	
	\section{Preliminaries} \label{section::AlgoIntro}
	We now state our problem formally; the optimisation terms (such as ``strongly convex'') are defined in our space below, and align with their usual meaning.

		\begin{definition}[UQ Problem]
		Let $\theta \in \Theta \subset \mathbb{R}^d$ be distributed according to $\pi$, and let $f : \mathbb{R}^n \times \mathbb{R}^d \rightarrow \mathbb{R}$ be accessible only through a function $F$ with $f(x,\theta) = \mathbb{E}_V\left(F(x,\theta,v)\right)$ and $\nabla f(x,\theta) = \mathbb{E}_V\left(\nabla F(x,\theta,v)\right)$, for example by dominated convergence. Let $f(\cdot, \theta)$ be differentiable, $\mu$-strongly convex and have an $L$-Lipschitz gradient for $\pi$-almost all $\theta$, with $\mu, L$ independent of $\theta$. We also assume there is some $x_0 \in L^2_\pi(\Theta;\mathbb{R}^n)$ such that $\theta \mapsto f(x_0(\theta),\theta)$ is $\pi$-integrable and $\theta \mapsto \nabla f(x_0(\theta),\theta)$ lies in $L^2_\pi(\Theta;\mathbb{R}^n)$. The Uncertainty Quantification problem is to find the minimiser $x^*$ of \eqref{eq:intro prob} over $L^2_\pi(\Theta;\mathbb{R}^n)$, and to learn the distribution of $x^*(\theta)$ when $\theta \sim \pi$, or at least its statistics.
	\end{definition}
    By the quadratic upper and lower bounds given by the Lipschitz gradient and strong convexity, the integrability properties assumed at $x_0$ then hold for every $x \in L^2_\pi(\Theta;\mathbb{R}^n)$.
    
	We also assume that for all $\theta$ and $x \in L^2_{\pi}$, the second moment of $\nabla F$ satisfies
	\begin{equation*}
	    \mathbb{E}_V \left( \nabla F(x(\theta),\theta,v)^2\right) \leq V + V_G \left(\nabla f(x(\theta),\theta)\right)^2,
	\end{equation*} 
	with $V \geq 0$, $V_G \geq 1$, and equality in the deterministic (with respect to $v$) case. This is used to bound the error of our gradient estimates in Section \ref{sec::IntegralError}.

	We choose a fixed real-valued orthonormal basis $B_i(\theta)$ of $L^2_{\pi}$, $i \geq 0$. For the Monte Carlo gradient estimates of Section \ref{sec::IntegralError}, we also require that $Q_m = \sup_\theta \sum_{i=0}^{m-1} |B_i(\theta)|^2 < \infty$ for every $m$, as in \cite{crepey2020uncertainty}. This holds for bounded bases, such as the Fourier basis or Legendre polynomials on a compact domain, but not for Hermite or Laguerre polynomials on unbounded domains. We then have $x^*(\theta) = \sum_{i \geq 0} u_iB_i(\theta)$ for some coefficients $u_i$, and our problem reduces to finding the $u_i$.
	
	The basis gives the canonical isomorphism $\mathcal{I}$ from $L^2_{\pi}$ to $l^2$, the space of square-summable sequences, with $u = \mathcal{I}(x)$ whenever $x(\theta) = \sum_i u_iB_i(\theta)$. By orthonormality, $\mathcal{I}$ is an isometry, $\langle x, y \rangle_\pi = \langle \mathcal{I}(x), \mathcal{I}(y) \rangle_2$ and $||x||_\pi = ||\mathcal{I}(x)||_2$, where $\langle x, y \rangle_\pi = \int_{\Theta} x(\theta)y(\theta) \pi(d\theta)$ and $\langle u, v \rangle_2 = \sum_{i=0}^\infty u_iv_i$. For $\mathbb{R}^n$-valued functions we apply $\mathcal{I}$ componentwise and sum the inner products over components, so that the isometry still holds. For a vector-valued quantity $z$, taking values in $\mathbb{R}^n$ or in $l^2$, we write $z^2$ for its squared Euclidean or $l^2$ norm $||z||_2^2$.

	\begin{definition}
	    We say $x \in L^2_{\pi,m}$, the level $m$ subspace of $L^2_\pi$, if only the first $m$ coefficients $u_0, \ldots, u_{m-1}$ of $u = \mathcal{I}(x)$ may be non-zero, and similarly write $u \in l^2_m$.
	\end{definition}
	Strong convexity and the Lipschitz continuity of the gradient are imposed with respect to the $\pi$ norm, and in the first argument only:
	\begin{definition} \label{defn::strongConvex}
		A function $f: \mathbb{R}^n \times \mathbb{R}^d \rightarrow \mathbb{R}$ is strongly convex with parameter $\mu$, and has an $L$-Lipschitz gradient, if for all $x_1(\cdot), x_2(\cdot) \in L^2_\pi$,
		\begin{gather*}
		    \langle \nabla f(x_1(\theta),\theta)- \nabla f(x_2(\theta),\theta), x_1(\theta) - x_2(\theta) \rangle_\pi \geq \mu||x_1(\theta) - x_2(\theta)||_\pi^2, \\
			||\nabla f(x_1(\theta), \theta) - \nabla f(x_2(\theta), \theta)||_\pi \leq L||x_1(\theta) - x_2(\theta)||_\pi .
		\end{gather*}
	\end{definition}
	Both properties then also hold on any level $m$ subspace $L^2_{\pi,m}$.
    
    Our iterates are elements of $l^2$, and in place of the gradient we use its basis coefficients.
    \begin{definition} \label{definition::basisGrad}
    The full gradient $D(f(x(\theta),\theta)) \defeq \mathcal{I}(\nabla f(x(\cdot),\cdot)) \in l^2$ has $i$th entry \newline $\int_{\Theta} \nabla f(x(\theta),\theta) B_i(\theta) \pi(d\theta)$. The truncated gradient $D_{m}(f(x(\theta),\theta)) \in l^2_m$ keeps its first $m$ entries and sets the rest to zero; equivalently, $-D_m$ is the direction of steepest descent within $l^2_m$ with respect to the $\pi$-norm.
    \end{definition}

    We fix our convention for indices. A subscript on a vector or function indexes its basis coefficients or components (as in $u_i$), a subscript on $\theta$ or $v$ indexes samples (as in $\theta_j$), and a subscript $m$ on an operator, constant or optimum indicates the truncation level (as in $P_m$, $Q_m$ or $u^*_m$). A superscript indexes the iteration and is used only for the iterates $u^k, x^k, y^k, z^k$ and $\xi^k$; all other iteration-dependent quantities, such as $\gamma_k$, $m_k$, $M_k$, $e_k$ and $\mathcal{F}_k$, carry the iteration as a subscript.
	\subsection{Condition numbers} \label{sec::ConditionNumbers}
	As usual, the condition number of the problem is $\kappa = L/\mu$, and gradient descent on a strongly convex function converges linearly with rate $1 - 1/\kappa$. Restricting $f$ to a level $m$ subspace $L^2_{\pi,m}$ can only improve the constants, but since our analysis involves $u^*$, which in general has infinitely many non-zero coefficients, our results use $\mu$ and $L$ throughout.

	We also introduce a condition number of the solution, which measures how many basis functions are needed to represent $x^*$ to a given accuracy.
	\begin{definition}
	    For $x \in L^2_\pi$, the level $m$ projection $P_m$ maps $L^2_\pi$ to $L^2_{\pi,m}$, and the level $m$ remainder is $R_m(x) \defeq \left(\text{Id} - P_m\right)x$. We use the same notation on $l^2$, where $P_m$ sets $u_i = 0$ for $i \geq m$, and write $R_{m} = \mathcal{I}(R_m(x^*)) = (\text{Id} - P_m)u^*$.
	\end{definition}
	\begin{definition}
		Let $f$ be strongly convex, with optimum $x^*$, and let $u^* = \mathcal{I}(x^*)$. The $\varepsilon$-condition number of the solution, $\kappa_{\varepsilon}$, is the smallest number of basis functions $m$ such that
		\begin{equation*}
		||R_m||_2 = ||\left(\text{Id} - P_m\right)u^*||_2 < \varepsilon .
		\end{equation*}
	\end{definition}
	We will see that $\kappa_\varepsilon$ determines how long it takes for linear convergence to begin.
	
	\section{Algorithm Formulation} \label{sec::Algo}

    We now detail our algorithm for gradient descent, presented in Algorithm \ref{alg::UQSpecific}. Since $u^*$ is infinite-dimensional, at iteration $k$ we truncate to a finite level $m_k$, setting $u_i = 0$ for all $i \geq m_k$ and updating only $u_0, \ldots, u_{m_k-1}$. We choose a non-decreasing sequence $m_k$ such that $||R_{m_k}||_2 \rightarrow 0$ as $k \rightarrow \infty$; in general, without prior knowledge of the form of the solution, this requires $m_k \rightarrow \infty$. At a fixed level $m_k$ the iterates converge to the \textit{level $m_k$ optimum} $u^*_{m_k}$, the minimiser over $l^2_{m_k}$. This is not in general equal to $P_{m_k}(u^*)$, but we show in Lemma \ref{lemma::LinConvRemainder} that the distance between the two converges to zero.

	The vector $D_{m_k} f(x^k(\theta),\theta)$ must be estimated. The most natural approach, which we use in our experiments, is a Monte Carlo approximation of the integrals $\int_{\Theta} \nabla f(x(\theta),\theta) B_i(\theta) \pi(d\theta)$, but our proofs only require the following assumptions on the estimate, and hold for any procedure that satisfies them.
    
    Let $\mathcal{F}_k$ be the $\sigma$-algebra generated by all random variables used to estimate the gradient in the first $k$ iterations, so that $x^k$, the point at which the $k$th gradient is estimated, is $\mathcal{F}_{k-1}$-measurable. Writing $D'_{m_k}(f(x^k(\theta),\theta)) = D_{m_k}(f(x^k(\theta),\theta)) + e_k$ for the noisy estimate of the truncated gradient of Definition \ref{definition::basisGrad}, we assume
    \begin{equation}
    \begin{aligned}
    &\mathbb{E}\left(e_k \mid \mathcal{F}_{k-1}\right) = 0, \qquad
    \mathbb{E}\left( \|D'_{m_k}(f(x^k(\theta),\theta))\|_2^2 \,\big\vert\, \mathcal{F}_{k-1} \right) \\
    &\qquad \leq C + C_G \|\nabla f(x^k(\theta),\theta) - \nabla f(x^*_{m_k}(\theta),\theta)\|_{\pi}^2,
    \end{aligned}
    \label{equation::noiseModel}
    \end{equation}
    where $\mathcal{I}(x^*_{m_k}) = u^*_{m_k}$.

\begin{algorithm}
\caption{Uncertainty Quantification for Gradient Descent}
\label{alg::UQSpecific}
\begin{algorithmic}
\STATE{Input: Function $f(x, \theta)$, a sequence of steps $\{\gamma_k\}$, a number of iterations $K$, a non-decreasing sequence of truncation points $\{m_k\}$, a number of Monte Carlo samples to estimate the gradient $\{M_k\}$, an initial vector $u^1 \in l^2$ with at most the first $m_1$ entries nonzero.}
\STATE{Output: A vector of coefficients $u^{K+1}$}
\FOR{$k = 1, \ldots, K$}
\STATE{Define $x^{k} = \mathcal{I}^{-1}(u^{k})$}.
\STATE{Estimate the gradient vector $D_{m_k}(f(x^{k}(\theta),\theta))$ as $D'_{m_k}(f(x^{k}(\theta),\theta))$}
\STATE{Update $u^{k+1} = u^{k} - \gamma_kD'_{m_k}(f(x^k(\theta), \theta))$}
\ENDFOR
\RETURN $u^{K+1}$
\end{algorithmic}
\end{algorithm}
	
\subsection{Two Important Properties}
	We now show that the truncated optima converge to the true optimum at an explicit rate, and that Monte Carlo gradient estimates satisfy the noise model \eqref{equation::noiseModel}, with explicit $C, C_G$.
	
	\subsubsection{Convergence of the optimum} \label{subsection::OptConvergence}
	
	In general the level $m$ optimum is not equal to the projection $P_m(u^*)$. We show that the distance between the two converges to zero at the rate of the remainder, with a constant depending on the condition number:
	\begin{lemma} \label{lemma::LinConvRemainder}
Let $u^*_m$ be the level $m$ optimum, let $x^*_m(\theta) = \mathcal{I}^{-1}(u^*_m)$, and let $\kappa$ be the condition number of $f$. Then
\[
  \|P_m(u^*) - u^*_m\|_2^2 \le \frac{(\kappa - 1)^2}{4\kappa}\,\|R_m\|_2^2,
\]
so $\|P_m(u^*) - u^*_m\|_2$ converges to zero at least as fast as $\|R_m\|_2$.
\end{lemma}

As $u^* = P_m(u^*) + R_m$ with the two terms orthogonal, this immediately implies
\begin{equation}\label{eq:truncated-optimum-full}
  \|u^* - u^*_m\|_2^2 \le \frac{(\kappa + 1)^2}{4\kappa}\,\|R_m\|_2^2.
\end{equation}

\begin{proof}
As in Definition \ref{definition::basisGrad}, we write $Df(x(\theta))$ for the vector of basis coefficients of the
gradient of $f$ at $x(\theta)$, suppressing the second argument. Let $g = Df(x^*_m(\theta))$ and $d = u^* - u^*_m$. By optimality,
$Df(x^*(\theta)) = 0$, and since $u^*_m$ minimises over the first $m$ coefficients, $P_m g = 0$. As $u^*_m$ is also zero beyond the first $m$
entries, $d = (P_m(u^*) - u^*_m) + R_m$ orthogonally, so $\|d\|_2^2 = \|P_m(u^*) - u^*_m\|_2^2 + \|R_m\|_2^2$.
Co-coercivity of strongly convex functions, which still holds in our setting (Lemma \ref{appendix::strongCocoercivity} in Appendix \ref{sec:suppConvexity}), together with the isometry of the $\pi$ and $2$ norms, gives
\[
    \begin{aligned}
    &(L + \mu)\, \bigl\langle Df(x^*(\theta)) - Df(x^*_m(\theta)), u^* - u^*_m \bigr\rangle_2 \\
    &\qquad \ge \mu L\,\|u^* - u^*_m\|_2^2 + \|Df(x^*(\theta)) - Df(x^*_m(\theta))\|_2^2,
    \end{aligned}
\]
which here reads $(L+\mu)\langle -g, d\rangle_2 \ge \mu L\|d\|_2^2 + \|g\|_2^2$. If $g = 0$, the
right-hand side forces $d = 0$ and the result is trivial, so assume $g \neq 0$. By the AM--GM
inequality, $\mu L\|d\|_2^2 + \|g\|_2^2 \ge 2\sqrt{\mu L}\,\|d\|_2\|g\|_2$, and hence
\begin{equation}\label{eq:kantorovich-angle}
  \Bigl\langle -\frac{g}{\|g\|_2},\, d \Bigr\rangle_2
  \ge \frac{2\sqrt{\mu L}}{L + \mu}\,\|d\|_2
  = \frac{2\sqrt{\kappa}}{\kappa + 1}\,\|d\|_2 .
\end{equation}
Since $P_m g = 0$, only the component of $d$ beyond the first $m$ entries contributes to the
left-hand side, so by Cauchy--Schwarz it is at most $\|(\mathrm{Id} - P_m)d\|_2 = \|R_m\|_2$.
Squaring \eqref{eq:kantorovich-angle} then gives
$\|R_m\|_2^2 \ge \frac{4\kappa}{(\kappa+1)^2}\bigl(\|P_m(u^*) - u^*_m\|_2^2 + \|R_m\|_2^2\bigr)$,
and rearranging, using $(\kappa+1)^2 - 4\kappa = (\kappa-1)^2$, gives the conclusion.
\end{proof}
Geometrically, \eqref{eq:kantorovich-angle} says that the angle between $u^* - u^*_m$ and the
gradient at the truncated optimum, which lies entirely beyond level $m$, is bounded away from $\pi/2$. We cannot use constants restricted to the level $m$
subspace here, as the inequalities involve $x^*$, which may have infinitely many non-zero basis coefficients.

\begin{corollary} \label{corol:gradLipschitzRem}
    With the notation of Lemma \ref{lemma::LinConvRemainder}, $\mu ||R_m||_2 \leq || Df(x^*_m(\theta))||_2 \leq L ||R_m||_2$.
\end{corollary}
\begin{proof}
    In the proof above, $P_m g = 0$ also gives $\langle -g, d \rangle_2 \le ||g||_2 ||R_m||_2$, and $||d||_2 \geq ||R_m||_2$. Substituting both into $(L+\mu)\langle -g, d\rangle_2 \ge \mu L\|d\|_2^2 + \|g\|_2^2$ gives $(||g||_2 - L ||R_m||_2) (||g||_2 - \mu ||R_m||_2) \leq 0$.
\end{proof}

    \subsubsection{Error Analysis of Noisy Gradient Information} \label{sec::IntegralError}
	The entries of $D_{m_k}$ must in practice be computed numerically. We use a simple Monte Carlo approximation: at iteration $k$ we independently sample $M_k$ values $\theta_j$ of $\theta$, and $P_j$ samples of $v$ for each $\theta_j$, giving the error
    \begin{equation} \label{equation::error}
    \begin{aligned}
    e_k^2
    &= \sum_{i=0}^{m_k-1}
    \Biggl(
    \frac{1}{M_k} \sum_{j=1}^{M_k} \frac{1}{P_j}
    \sum_{p=1}^{P_j}
    \nabla F(x^k(\theta_j),\theta_j,v_{j,p})
    B_i(\theta_j) \\
    &\qquad\quad
    - \int \nabla f(x^k(\theta),\theta)
    B_i(\theta)\,\pi(d\theta)
    \Biggr)^2 .
    \end{aligned}
    \end{equation}
	The following bound proceeds similarly to \cite{crepey2020uncertainty}.
	\begin{theorem} \label{theorem::IntegralError}
	    Consider Algorithm \ref{alg::UQSpecific} with the Monte Carlo estimate above, and suppose that, conditionally on $\mathcal{F}_{k-1}$, $\theta_1, \dots, \theta_{M_k}$ are i.i.d.\ from $\pi$ and, for each $j$, $v_{j,1}, \dots, v_{j,P_j}$ are i.i.d.\ from $\nu(\theta_j, \cdot)$, independently across $j$, with sample sizes $M_k$ and $P_1, \dots, P_{M_k}$ fixed in advance. Let $\bar{P}_k$ be the harmonic mean of $P_1, \ldots, P_{M_k}$, that is $\frac{1}{\bar{P}_k} \defeq \frac{1}{M_k}\sum_{j=1}^{M_k}\frac{1}{P_j}$, and set $\hat{V}_k \defeq V/\bar{P}_k$, $\hat{V}_{G,k} \defeq 1 + (V_G-1)/\bar{P}_k$. Then
		\begin{equation*}
		\mathbb{E}(||e_k||_2^2 \vert \mathcal{F}_{k-1}) \leq  \frac{Q_{m_k}}{M_k} \left(\hat{V}_{G,k} ||\nabla f(x^k(\theta), \theta)||_\pi^2 + \hat{V}_k \right),
		\end{equation*}
		where $x^k = \mathcal{I}^{-1}(u^k)$. If $P_j \equiv 1$, then $\hat{V}_k = V$ and $\hat{V}_{G,k} = V_G$.
	\end{theorem}
	The proof is given in Appendix \ref{sec:suppProofs}. It splits each coefficient error into the error from sampling $v$ at the fixed points $\theta_j$, and the Monte Carlo error over $\theta$. The two are uncorrelated conditionally on $\mathcal{F}_{k-1}$, and each is bounded using the second-moment assumption on $\nabla F$ and the definition of $Q_{m_k}$.

	\begin{corollary} \label{corol::altRepr}
		The bound of Theorem \ref{theorem::IntegralError} can also be written as
    \begin{equation}
    \begin{aligned}
    \mathbb{E}\left(\|e_k\|_2^2 \mid \mathcal{F}_{k-1}\right)
    \leq\;&
    \frac{Q_{m_k}}{M_k}
    \left(
    2\hat{V}_{G,k}
    \left\|\nabla f(x^*_{m_k}(\theta), \theta)\right\|_\pi^2
    + \hat{V}_k
    \right) \\
    &+ \frac{2Q_{m_k}}{M_k}\hat{V}_{G,k}
    \left\|
    \nabla f(x^k(\theta), \theta)
    - \nabla f(x^*_{m_k}(\theta), \theta)
    \right\|_\pi^2.
    \end{aligned}
    \end{equation}
	\end{corollary}
	\begin{proof}
	    Add and subtract $\nabla f(x_{m_k}^*(\theta), \theta)$, and use $(a+b)^2 \leq 2a^2 + 2b^2$.
	\end{proof}
    The term $\|\nabla f(x^*_{m_k}(\theta),\theta)\|_\pi^2$ is the full $\pi$-norm of the gradient, and so need not vanish. However, $D_{m_k}f(x^*_{m_k}(\theta),\theta) = 0$, so it is due entirely to basis functions beyond level $m_k$, and is at most $L^2 \|R_{m_k}\|_2^2$ by Corollary \ref{corol:gradLipschitzRem}.
	\begin{corollary} \label{corol:Noisemodel}
		At each iteration $k$, Algorithm \ref{alg::UQSpecific} satisfies the noise model \eqref{equation::noiseModel} with
		\begin{equation}
			C = \frac{Q_{m_k}}{M_k} \left(2\hat{V}_{G,k}||\nabla f(x^*_{m_k}(\theta), \theta)||_\pi^2 + \hat{V}_k\right), \qquad C_G = 1 + \frac{2\hat{V}_{G,k} Q_{m_k}}{M_k}.
		\end{equation}
	\end{corollary}
	\begin{proof}
	Write $g_k = D_{m_k}(f(x^k(\theta),\theta))$ and $g'_k = g_k + e_k$. Since $g_k$ is $\mathcal{F}_{k-1}$-measurable and $\mathbb{E}(e_k | \mathcal{F}_{k-1}) = 0$, $\mathbb{E}((g'_k)^2 | \mathcal{F}_{k-1}) = g_k^2 + \mathbb{E}(e_k^2 | \mathcal{F}_{k-1})$. As $D_{m_k}(f(x^*_{m_k}(\theta),\theta)) = 0$, Bessel's inequality gives $g_k^2 \leq ||\nabla f(x^k(\theta), \theta) - \nabla f(x^*_{m_k}(\theta), \theta)||_\pi^2$, and we conclude with Corollary \ref{corol::altRepr}.
	\end{proof}
	The quantity $Q_{m_k}$ grows with $m_k$, and enters $C$ and $C_G$ only through $Q_{m_k}/M_k$. The first part of $C$ is at most $\frac{2Q_{m_k}}{M_k}\hat{V}_{G,k}L^2||R_{m_k}||_2^2$, so it vanishes as $m_k \rightarrow \infty$ as long as $Q_{m_k}/M_k$ stays bounded. The second part, $Q_{m_k}\hat{V}_k/M_k$, and the term $2\hat{V}_{G,k}Q_{m_k}/M_k$ in $C_G$, can only be controlled through the sample sizes. Keeping $C_G$ bounded, which is needed for a constant step size in Theorem \ref{theorem::GDFixedUQ}, requires $M_k$ to grow at least proportionally to $Q_{m_k}$; the second part of $C$ then vanishes automatically when $V = 0$, and otherwise only if $M_k\bar{P}_k$ grows faster than $Q_{m_k}$. We make these trade-offs precise in Section \ref{sec::GD}.

	\section{Gradient Descent} \label{sec::GD}
	We now give a convergence rate for Algorithm \ref{alg::UQSpecific}. The iterate $u^k$ has all entries from $m_k$ onwards equal to zero, and is $\mathcal{F}_{k-1}$-measurable. By orthogonality,
	\begin{equation}
	||u^k-u^*||_2^2 = ||u^k - P_{m_k}(u^*)||_2^2 + ||R_{m_k}||_2^2 . \label{basicconvone}
	\end{equation}
    For typical bases, a regularity condition gives an explicit rate of decay of $\|R_{m_k}\|_2$ in $m_k$. For example, for the tensor-product Legendre basis on $[-1,1]^d$, suppose that $x^* \in H^p([-1,1]^d)$ for some $p \geq 0$, that is, $x^*$ and its derivatives up to order $p$ are square-integrable, and that the first $m = (N+1)^d$ basis functions span the polynomials of degree at most $N$ in each variable. Then $\|R_m\|_2 \leq C N^{-p} \|x^*\|_{H^p} \leq C' m^{-p/d} \|x^*\|_{H^p}$, where $C, C'$ depend only on $p$ and $d$ \cite[Ch.~5]{canuto2006spectral}.
    
	\begin{exmp} \label{exmp::VarianceExplain}
		The natural alternative is to fix a large number of basis functions $m$ throughout, say in the Fourier basis. A good $m$ is not known in advance: too small a level leaves truncation error $||R_m||_2$, while too large a level gives a large $Q_m$, so by Corollary \ref{corol:Noisemodel} either the step size must shrink or $M_k$ must grow, even when the error is dominated by low-order coefficients. Growing the basis avoids both, and in effect reduces the variance of early gradient estimates without extra samples. We quantify the cost of a fixed level after Corollary \ref{cor:complexity} and illustrate it in Section \ref{sssec:varRed}.
	\end{exmp} 

	The remainder $||R_{m_k}||_2$ does not depend on the iterates, so we focus on the first term of \eqref{basicconvone}, which is at most $2||u^k - u^*_{m_k}||_2^2 + 2||u^*_{m_k} - P_{m_k}(u^*)||_2^2$. Lemma \ref{lemma::LinConvRemainder} controls the second term, and we now show that the first converges linearly at a fixed level.
	
	\subsection{Fixed Level Analysis} \label{sec::truncGrad}
	Throughout this subsection the level $m$ is fixed.
	\begin{theorem} \label{theorem::GDFixedUQ}
		Fix a level of basis functions $m$, and a step size $\gamma$. Let the filtration $\mathcal{F}_k$ be generated by the random variables used to estimate the gradient in the first $k$ steps, so that $u^k$ is $\mathcal{F}_{k-1}$-measurable, and let $u^*_m$ denote the level $m$ optimum. Suppose we observe an unbiased estimate $g'_k$ of $g_k = D_m f(x^k(\theta),\theta)$, in such a way that:
		\begin{equation*}
		\mathbb{E}(||g'_k||_2^2 | \mathcal{F}_{k-1}) \leq C + C_G||\nabla f(x^k(\theta), \theta) - \nabla f(x_{m}^*(\theta), \theta)||_{\pi}^2
		\end{equation*}
		with $C_G \geq 1, C \geq 0$. Then we have:
		\begin{align*}
		\mathbb{E}\left(||u^{k+1} - u^*_m||_{2}^2\right) &\leq \left(1 - \frac{2\gamma \mu L}{\mu + L} \right)^k \left(||u^1 - u^*_m||_{2}^2 - \frac{(\mu+L)\gamma C}{2 \mu L}\right) \nonumber \\
		&+ \frac{(\mu+L)\gamma C}{2 \mu L}
		\end{align*}
		where $\gamma \leq \frac{2}{\left(\mu + L\right)\left(C_G\right)}$, and $u^1 \in l^2_{m}$ is deterministic.
	\end{theorem}
	Note that $u^*_m \in l^2_m$ and $g_k, g'_k \in l^2_m$ for all $k$, whereas $u^*$ itself will in general have infinitely many non-zero coefficients.
	The proof is given in Appendix \ref{sec:suppProofs}. It is the standard one-step argument for gradient descent on a strongly convex function: we expand $\mathbb{E}(||u^{k+1} - u^*_m||_2^2 | \mathcal{F}_{k-1})$, bound the inner product term from below using co-coercivity of the gradient, and bound the second moment of $g'_k$ using the noise model. The step size condition makes the remaining gradient-difference term non-positive.

	By Corollary \ref{corol:Noisemodel}, the step size condition of Theorem \ref{theorem::GDFixedUQ} reads $\gamma \leq 2/((\mu+L)C_G)$ with $C_G = 1 + 2\hat{V}_{G,k}Q_{m_k}/M_k$. Since $Q_{m_k}$ grows with $m_k$, the maximal step size shrinks as $m_k$ grows unless $M_k$ grows at least proportionally to $Q_{m_k}$. While running Algorithm \ref{alg::UQSpecific} we can therefore either select the maximal step size permitted for each $m_k$ with $M_k$ fixed, or keep the step size constant by increasing $M_k$ with $Q_{m_k}$; the latter is the setting of Theorem \ref{theorem::gradConvLinear}.
	\subsection{Linear Convergence Phase}
	Here we analyse the linearly convergent phase, which occurs when the error of $u^k$ to the optimum is dominated by the $u^k - P_{m_k}(u^*)$ term, and $||R_{m_k}||_2$ is small. In this subsection, we show that once $k$ is high enough, $\mathbb{E}||u^k - u^*||_2$ decreases linearly until it reaches a neighbourhood of zero, whose radius consists of a part controlled by $||R_{m_k}||_2$ and a Monte Carlo part that depends only on $V$.
	
    By Lemma \ref{lemma::LinConvRemainder} and Corollary \ref{corol:gradLipschitzRem},
    $||P_{m_k}(u^*) - u^*_{m_k}||_2 \leq \sqrt{\kappa}\,||R_{m_k}||_2$,
    $||u^* - u^*_{m_k}||_2 \leq \sqrt{\kappa+1}\,||R_{m_k}||_2$ and
    $||\nabla f(x^*_{m_k}(\theta),\theta)||_\pi \leq L\,||R_{m_k}||_2$, so every
    basis dependent quantity below is controlled by $||R_{m_k}||_2$. We also write
    \begin{equation}
    \sigma_k \defeq \gamma\sqrt{\frac{Q_{m_k}\hat{V}_k}{M_k}}
    \label{equation::mcFloor}
    \end{equation}
    for the part of the gradient noise that does not vanish at the optimum.
    \begin{theorem} \label{theorem::gradConvLinear}
    Consider Algorithm \ref{alg::UQSpecific} with constant step size $0 < \gamma < \frac{2}{\mu + L}$,
    and suppose that for all $k \geq K$,
    \begin{equation*}
    M_k \geq \frac{2\gamma Q_{m_k}\hat{V}_{G,k}}{2/(\mu+L) - \gamma}
    \qquad \text{and} \qquad ||R_{m_k}||_2 \leq \varepsilon .
    \end{equation*}
    Let $\rho = 1 - \frac{2\gamma\mu L}{\mu+L}$,
    $\bar\sigma = \sup_{k\geq K}\sigma_k$ and
    $c_R = (1+\rho^{1/2})\sqrt{\kappa+1} + L\sqrt{2\gamma/(\mu+L)}$.
    Then for all $N \geq 0$,
    \begin{equation*}
    \mathbb{E}||u^{K+N} - u^*||_2 \leq \rho^{N/2}\,\mathbb{E}||u^K - u^*||_2
    + \frac{c_R\varepsilon + \bar\sigma}{1 - \rho^{1/2}} .
    \end{equation*}
    Moreover $\bar\sigma^2 \leq \frac{V\gamma}{\mu+L}$, and $\bar\sigma = 0$
    when $V = 0$.
    \end{theorem}
    \begin{proof}
    Fix $k \geq K$. The step from $u^k$ to $u^{k+1}$ is a gradient step at
    level $m_k$ from $u^k \in l^2_{m_k}$, and the condition on $M_k$ is exactly
    $\gamma \leq 2/((\mu+L)C_G)$ for the constants of Corollary
    \ref{corol:Noisemodel}. The one-step inequality in the proof of Theorem
    \ref{theorem::GDFixedUQ} gives
    \begin{equation*}
    \mathbb{E}\left(||u^{k+1} - u^*_{m_k}||_2^2 | \mathcal{F}_{k-1}\right)
    \leq \rho\,||u^k - u^*_{m_k}||_2^2 + \gamma^2 C ,
    \end{equation*}
    so by conditional Jensen,
    $\mathbb{E}(||u^{k+1} - u^*_{m_k}||_2 | \mathcal{F}_{k-1})
    \leq \rho^{1/2}||u^k - u^*_{m_k}||_2 + \gamma\sqrt{C}$.
    The condition on $M_k$ gives
    $\gamma^2 \frac{2 Q_{m_k}\hat{V}_{G,k}}{M_k} \leq \frac{2\gamma}{\mu+L}$,
    hence $\gamma\sqrt{C} \leq L\sqrt{2\gamma/(\mu+L)}\,||R_{m_k}||_2 + \sigma_k$.
    Using the triangle inequality through $u^*_{m_k}$ on both sides and taking
    expectations,
    \begin{equation*}
    \mathbb{E}||u^{k+1} - u^*||_2 \leq \rho^{1/2}\,\mathbb{E}||u^k - u^*||_2
    + c_R\,||R_{m_k}||_2 + \sigma_k .
    \end{equation*}
    Iterating from $k = K$ and summing the geometric series gives the result.
    \end{proof}
    \begin{corollary} \label{corol::exactLinear}
    In the setting of Theorem \ref{theorem::gradConvLinear}, suppose instead that
    $||R_{m_k}||_2 \leq A r^k$ and $\sigma_k \leq B r^k$ for all $k \geq K$,
    for some $r \in (0,1)$ with $r \neq \rho^{1/2}$. Then
    \begin{equation*}
    \mathbb{E}||u^{K+N} - u^*||_2 \leq \left(\mathbb{E}||u^K - u^*||_2
    + \frac{(c_R A + B)\, r^K}{|\rho^{1/2} - r|}\right)
    \max(\rho^{1/2}, r)^N .
    \end{equation*}
    \end{corollary}
    \begin{proof}
        As in the previous Theorem, unroll the recursion and sum the geometric series.
    \end{proof}
	Theorem \ref{theorem::gradConvLinear} shows that the error eventually lies within $\frac{c_R\varepsilon + \bar\sigma}{1-\rho^{1/2}}$ of $u^*$, up to a geometrically decaying term. To reach a target accuracy $\delta$, that is $\limsup_{N\to\infty}\mathbb{E}||u^{K+N}-u^*||_2 \leq \delta$, we therefore need $\bar\sigma < (1-\rho^{1/2})\delta$, and it then suffices that $||R_{m_k}||_2 \leq \varepsilon^*$ for all $k \geq K$, where
	\begin{equation}
	    \varepsilon^* = \frac{(1-\rho^{1/2})\delta - \bar\sigma}{c_R}.
	\end{equation}
	In the language of Section \ref{sec::ConditionNumbers}, it suffices that $m_k \geq \kappa_{\varepsilon^*}$, the $\varepsilon^*$-condition number of the solution. The two condition numbers are therefore linked: since $1-\rho^{1/2} \leq 1-\rho = \frac{2\gamma\mu L}{\mu+L} \leq \frac{4}{\kappa}$ (using $\gamma < \frac{2}{\mu+L}$) and $c_R \geq \sqrt{\kappa+1}$, we have $\varepsilon^* \leq 4\kappa^{-3/2}\delta$, so a worse-conditioned problem requires a smaller remainder, and hence more basis functions, to reach the same accuracy. This is another motivating reason for our consideration of the accelerated gradient descent, which we will move onto in the next section.

    We now compare our constant step sizes with those of \cite{crepey2020uncertainty}. Their Assumption A2 requires, among other conditions, $\sum_k \gamma_k = \infty$, $\sum_k \gamma_k^{1+c} < \infty$ for some $c \in (0,1]$, and $\sum_k \gamma_k^2 Q_{m_k}/M_k < \infty$. The second condition forces $\gamma_k \rightarrow 0$ for any $M_k$, although constant sample sizes are admissible, whereas we need $M_k$ to grow with $Q_{m_k}$ to keep our step size constant. At equal sample sizes, however, the first and third conditions prevent $\gamma_k Q_{m_k}/M_k$ from staying bounded away from zero, while our maximal step size $2/((\mu+L)C_G)$ keeps it bounded away from zero, so their step sizes become arbitrarily small compared with ours. If $M_k$ grows proportionally to $Q_{m_k}$, our step size is constant while theirs tends to zero, and if $M_k$ is fixed, ours decays like $1/Q_{m_k}$. In both cases our step size decreases only because the basis grows, not with the iteration count, and at fixed $m_k$ and $M_k$ we can iterate with a constant step size to convergence to a neighbourhood.

    \subsection{Convergence to $u^*$ and total cost} \label{ss::totalCost}
    In Corollary \ref{corol::exactLinear}, we saw linear convergence to $u^*$ itself, but only when the remainder decays geometrically in $k$. Here we treat a general remainder, along with the costs of this linear convergence. First, an obvious limitation from Equation \eqref{basicconvone} is:
    \begin{equation} \label{eq::remainderLowerBound}
        ||u^k - u^*||_2 \geq ||R_{m_k}||_2 ,
    \end{equation}
    that is, the total error can decay no faster than the remainder. The decay of the remainder depends on the regularity of $x^*$ in $\theta$, not the constants $\mu, L$. This is why we introduced the condition number of the solution in Section \ref{sec::ConditionNumbers}.

    The following result shows that, up to a constant, Algorithm \ref{alg::UQSpecific} attains the lower bound \eqref{eq::remainderLowerBound} whenever the remainder does not decay faster than the optimisation error. Corollary \ref{corol::exactLinear} covers the opposite case.

    \begin{proposition} \label{prop::tracking}
    In the setting of Theorem \ref{theorem::gradConvLinear}, write $\varepsilon_k = ||R_{m_k}||_2$ and suppose that $\varepsilon_k > 0$ and that there are constants $\Lambda \geq 1$ and $1 \leq \lambda < \rho^{-1/2}$ with
    \begin{equation} \label{eq::slowRemainder}
        \varepsilon_k \leq \Lambda \lambda^{l-k}\varepsilon_l \qquad \text{for all } K \leq k \leq l .
    \end{equation}
    Then for all $N \geq 1$,
    \begin{equation*}
        ||R_{m_{K+N}}||_2 \leq \mathbb{E}||u^{K+N} - u^*||_2 \leq \rho^{N/2}\,\mathbb{E}||u^K - u^*||_2 + \frac{\Lambda\lambda\, c_R}{1 - \lambda\rho^{1/2}}\,||R_{m_{K+N}}||_2 + \frac{\bar\sigma}{1-\rho^{1/2}} .
    \end{equation*}
    \end{proposition}
    \begin{proof}
    The lower bound comes from Equation \eqref{eq::remainderLowerBound} we stated at the beginning of this subsection. For the upper bound, the proof of Theorem \ref{theorem::gradConvLinear} gives $e_{k+1} \leq \rho^{1/2}e_k + c_R\varepsilon_k + \sigma_k$ for $k \geq K$, where $e_k = \mathbb{E}||u^k - u^*||_2$. Unrolling,
    \begin{equation*}
        e_{K+N} \leq \rho^{N/2}e_K + \sum_{j=0}^{N-1}\rho^{j/2}\left(c_R\,\varepsilon_{K+N-1-j} + \sigma_{K+N-1-j}\right).
    \end{equation*}
    By Equation \eqref{eq::slowRemainder}, $\varepsilon_{K+N-1-j} \leq \Lambda\lambda^{j+1}\varepsilon_{K+N}$, so the remainder terms sum to at most \newline $\Lambda\lambda c_R\,\varepsilon_{K+N}\sum_{j\geq 0}(\lambda\rho^{1/2})^j$. The noise terms sum to at most $\bar\sigma/(1-\rho^{1/2})$. The conclusion follows from summing the geometric series in the remainder terms.
    \end{proof}

    In the case that $V = 0$ (and therefore $\bar\sigma = 0$), this implies that once the initial error has been forgotten, the error in the iterate is within a constant factor of the remainder. That is, the iterates are as accurate as possible up to a constant. The condition in Equation \eqref{eq::slowRemainder} holds, for example:
    \begin{itemize}
        \item when $c\,m^{-p} \leq ||R_m||_2 \leq C m^{-p}$ and $m_k$ grows polynomially in $k$, as in all the schedules of Section \ref{sec::Experiments}. Here $\Lambda$ is of order $C/c$, and any $\lambda > 1$ is allowed once $K$ is large enough. The total error then decays polynomially in $k$, at the rate set by the smoothness of $x^*$ and the growth of $m_k$, and by \eqref{eq::remainderLowerBound} no faster rate is possible with that schedule;
        \item when $||R_{m_k}||_2 = A r^k$ with $\rho^{1/2} < r < 1$, with $\Lambda = 1$ and $\lambda = 1/r$. Together with Corollary \ref{corol::exactLinear}, this shows that the total error decays like the slower of $\rho^{k/2}$ and the remainder.
    \end{itemize}
    When $V > 0$, the term $\bar\sigma/(1-\rho^{1/2})$ does not shrink as the basis grows. Smaller neighbourhoods then require the sample sizes to grow as well, as in Corollary \ref{corol::exactLinear}.

    We note that we can always ensure that the remainder has geometric decay by having $m_k$ grow geometrically. However, in this case, iterations grow in cost as $M_k$ must also grow with $Q_{m_k}$. The iteration count is therefore not the right measure of cost, and we now count gradient evaluations. 

    \begin{corollary}[Gradient evaluations to accuracy $\delta$] \label{cor:complexity}
    Suppose $V = 0$, and that there are constants $A, B > 0$, $p > 0$ and $q > 0$ such that
    \begin{equation} \label{eq:regularity}
    \|R_m\|_2 \le A\, m^{-p}
    \quad\text{and}\quad
    Q_m \le B\, m^{q}
    \qquad\text{for all } m \ge 1 .
    \end{equation}
    Run Algorithm \ref{alg::UQSpecific} with the constant step size $\gamma = 1/(\mu + L)$, an initial point $u^1 \in l^2_{m_1}$, and
    \begin{equation} \label{eq:geometric-schedule}
    m_k = \big\lceil \rho^{-k/(4p)} \big\rceil,
    \qquad
    M_k = \big\lceil 2 V_G\, Q_{m_k} \big\rceil .
    \end{equation}
    Let
    \[
    D = \|u^1 - u^*\|_2 + 8 \kappa A \big( 2\sqrt{\kappa + 1} + \sqrt{2} \big).
    \]
    Then $\mathbb{E}\|u^{N+1} - u^*\|_2 \le D \rho^{N/4}$ for all $N \ge 0$. Consequently, for any $\delta \in (0, D)$, the accuracy $\mathbb{E}\|u^{N+1} - u^*\|_2 \le \delta$ is reached after at most
    \begin{equation} \label{eq:complexity}
    W(\delta) \;\le\; 2^{\,4 + q(1 + 1/(2p))}\, V_G B\, \frac{p}{q}\, \kappa
    \left( \frac{D}{\delta} \right)^{q/p}
    \;+\; 8 \kappa \log\frac{D}{\delta} \;+\; 1
    \end{equation}
    evaluations of $\nabla F$. In particular
    $W(\delta) = O\big( \tfrac{p}{q}\, V_G B\, \kappa\, (A \kappa^{3/2} + \|u^1 - u^*\|_2)^{q/p}\, \delta^{-q/p} \big)$
    as $\delta \to 0$.
    \end{corollary}

    The proof is given in Appendix \ref{sec:suppProofs}. It applies Corollary \ref{corol::exactLinear} with $r = \rho^{1/4}$, which gives the bound on $\mathbb{E}\|u^{N+1} - u^*\|_2$, and then sums the per-iteration costs $M_k$, which grow geometrically in $k$.

    Earlier, when motivating this method, we talked about potential benefits of using the UQ procedure over simply setting $m_k$ to be high and constant. We now discuss that again with the hindsight of the above corollary. Suppose the level is fixed at $m$, with $V = 0$ and $M_k = \lceil 2 V_G Q_M \rceil $. By Theorem \ref{theorem::gradConvLinear} reaching accuracy $\delta$ requires:
    \begin{equation*}
        c_R \|R_m\|_2 / (1 - \rho^{1/2}) < \delta.
    \end{equation*}
    Since $1 - \rho^{1/2} = O(1/\kappa)$ and $c_R = O(\sqrt{\kappa})$, this forces $m \gtrsim (A\kappa^{3/2}/\delta)^{1/p}$. It also requires about $\kappa \log(1/\delta)$ iterations, each costing order $V_G B m^q$. The corresponding bound is therefore of order
    \[
    V_G B\, \kappa \log(1/\delta) \left( \frac{A \kappa^{3/2}}{\delta} \right)^{q/p},
    \]
    which is larger than \eqref{eq:complexity} by a factor of order $\log(1/\delta)$, ignoring the dependence on $p/q$. Additionally, note that the fixed level must be chosen in advance, and our schedule does not require knowledge of  $\delta$ or $A$.

    The schedule in Equation \eqref{eq:geometric-schedule} requires $\kappa$ and the smoothness exponent $p$. The schedules in Section \ref{sec::Experiments} are polynomial in $k$ rather than geometric, and the adaptive rule of \cite{dong2025beyond} needs neither $p$ nor $\kappa$. Achieving the bound \eqref{eq:complexity} without knowledge of $p$, for instance with a stall rule, is the open question raised in Section \ref{ss:choosingMk}.

    If $V > 0$, Corollary \ref{corol::exactLinear} also needs $\sigma_k \leq B' r^k$, that is $M_k\bar{P}_k/Q_{m_k} \gtrsim r^{-2k}$. The sample sizes must then grow geometrically faster than $Q_{m_k}$, and the same argument gives a total cost of order $\delta^{-2-q/p}$, up to factors depending on $\kappa$.

	\section{Accelerated Gradient} \label{sec::AGD}

    In this section we show that the robust accelerated gradient method of \cite{aybat2019universally}, run for one stage only, converges at an accelerated rate when the number of basis functions grows. At a fixed level, this follows from \cite[Theorem K.1]{aybat2019universally}, which already allows a state-dependent noise model; we state it in our notation as Theorem \ref{theorem::AccGradLevel}, and in Appendix \ref{sec:suppAGDFixed} we check that each step of its proof holds conditionally on $\mathcal{F}_{k-1}$ and that our gradient estimates satisfy its noise model. The new ingredient is Lemma \ref{lemma::levelChangeLyap}, which bounds the change in the Lyapunov function when basis functions are added, and leads to Theorem \ref{theorem::AccGradOverall}. Until the linear convergence phase below, we work at a fixed level $m$, and write $u^*_m$ and $x^*_m = \mathcal{I}^{-1}(u^*_m)$ for the level $m$ optimum. Throughout this section, we write $\bar{f}(x) \defeq \int_{\Theta} f(x(\theta),\theta)\,\pi(d\theta)$ for the $\pi$-expectation, or \textit{mean}, of $f$.

\begin{algorithm}
    \caption{Uncertainty Quantification for Accelerated Gradient Descent}
    \label{algo_accel}
        \begin{algorithmic}
        \STATE{Input: Function $f(x, \theta)$, two parameters $\alpha, \beta$, a number of iterations $K$, a non-decreasing sequence of truncation points $\{m_k\}$, sample sizes $\{M_k\}$, and an initial vector $u^1 \in l^2$}
        \STATE{Output: A vector of coefficients $u$.}
        \STATE{For all $j \geq m_1$, let $u_j^1 = 0$.}
        \STATE{Let $u^{0} = u^{1}$}
        \FOR{$k = 1, \ldots, K$}
        \STATE{Set $y^k = (1+\beta)u^k - \beta u^{k-1}$}
        \STATE{Set $z^k = \mathcal{I}^{-1}(y^k)$}
        \STATE{Estimate the gradient vector $D_{m_k}(f(z^{k}(\theta),\theta))$ as $D'_{m_k}(f(z^{k}(\theta),\theta))$}
        \STATE{$u^{k+1} = y^{k} - \alpha D'_{m_k}(f(z^{k}(\theta), \theta))$}
        \ENDFOR
        \RETURN $u^{K+1}$
        \end{algorithmic}
\end{algorithm}
Unlike the noise model of Equation \eqref{equation::noiseModel}, which we used for gradient descent and which bounds the second moment of the gradient estimate, here we bound the noise \newline $e_k = D'_m(f(z^k(\theta),\theta)) - D_m(f(z^k(\theta),\theta))$ itself:
\begin{equation}
\mathbb{E}\left(e_k | \mathcal{F}_{k-1}\right) = 0, \quad \mathbb{E}\left(||e_k||_2^2 | \mathcal{F}_{k-1}\right) \leq C + C_G||y^k - u^*_m||_2^2\label{equation::errVarianceAccel}
\end{equation}
This is the noise model of \cite[Appendix K]{aybat2019universally}.
Corollary \ref{corol::altRepr}, whose proof only uses that the point at which the gradient is estimated is $\mathcal{F}_{k-1}$-measurable, applies at $y^k$. Combined with $||\nabla f(z^k(\theta),\theta) - \nabla f(x^*_m(\theta),\theta)||_\pi \leq L||y^k - u^*_m||_2$, it shows that Algorithm \ref{algo_accel} satisfies Equation \eqref{equation::errVarianceAccel} with
\begin{equation*}
	C = \frac{Q_{m}}{M_k}\left(2\hat{V}_{G,k}||\nabla f(x^*_m(\theta), \theta)||_\pi^2 + \hat{V}_k\right), \qquad C_G = \frac{2L^2 \hat{V}_{G,k} Q_{m}}{M_k}.
\end{equation*}
Unlike in Corollary \ref{corol:Noisemodel}, there is no additive $1$ in $C_G$, since Equation \eqref{equation::errVarianceAccel} bounds the noise $e_k$ itself rather than the second moment of the estimate.

Following \cite{aybat2019universally}, we write the method as a linear dynamical system in the state $\xi^k = [(u^k)^T \quad (u^{k-1})^T]^T$ (see Appendix \ref{sec:suppAGDDynamics}), with $T = [I \quad 0]$ so that $T\xi^k = u^k$, and $\xi^* = [(u^*_m)^T \quad (u^*_m)^T]^T$. For a symmetric positive semi-definite matrix $P$, we use the Lyapunov function
\begin{equation*}
V_P(\xi) = (\xi - \xi^*)^T P(\xi-\xi^*) + h(\xi), \qquad h(\xi) = \bar{f}(\mathcal{I}^{-1}(T\xi)) - \bar{f}(x^*_m),
\end{equation*}
where $h$ is non-negative for $\xi \in (l^2_m)^2$ and vanishes at $\xi = \xi^*$.
\begin{theorem}[{\cite[Theorem K.1]{aybat2019universally}}] \label{theorem::AccGradLevel}
	Suppose $\kappa \geq 4$, and consider Algorithm \ref{algo_accel} at a fixed level $m$, that is with $m_k = m$ for all $k$, under the noise model \eqref{equation::errVarianceAccel}. Let $\alpha \in \left(0, \bar{\alpha}\right]$, where $\bar{\alpha} = \min\left\{\frac{1}{L}, \frac{\mu^3}{(60C_G)^2}\right\}$, and $\beta = \frac{1 - \sqrt{\alpha \mu}}{1 + \sqrt{\alpha \mu}}$. Let $Q_{\alpha} = \tilde{Q}_{\alpha} \otimes I$, where
	\begin{equation*}
	\tilde{Q}_{\alpha} = p_\alpha p_\alpha^T + 2\alpha C_G\, c\, c^T, \qquad p_\alpha = \begin{pmatrix} \sqrt{\frac{1}{2\alpha}} \\ \sqrt{\frac{\mu}{2}} - \sqrt{\frac{1}{2\alpha}} \end{pmatrix}, \qquad c = \begin{pmatrix} 1+\beta \\ -\beta \end{pmatrix}.
	\end{equation*}
	Then
	\begin{equation*}
	\mathbb{E}(V_{Q_\alpha}(\xi^{k+1})| \mathcal{F}_{k-1}) \leq \left(1 - \frac{\sqrt{\alpha \mu}}{3} \right)V_{Q_\alpha}(\xi^{k}) + 2C\alpha .
	\end{equation*}
\end{theorem}
The proof, which follows that of \cite[Theorem K.1]{aybat2019universally}, and the proof of the following corollary, are given in Appendix \ref{sec:suppAGDFixed}.
\begin{corollary} \label{corol:UQAGDFixed}
	In the set up of Theorem \ref{theorem::AccGradLevel}, we have:
	\begin{align*}
	\mathbb{E}||u^{k+1} - u^*_m||_2^2
	&\leq \left(1 - \frac{\sqrt{\alpha \mu}}{3}\right)^k \left(1 + \kappa + \frac{4\alpha C_G}{\mu}\right) ||u^1 - u^*_m||_2^2 \nonumber \\
	&+ \frac{12\, C\sqrt{\alpha}}{\mu^{3/2}}
	\end{align*}
\end{corollary}
The step size condition of Theorem \ref{theorem::AccGradLevel} is restrictive: substituting our $C_G$, the choice $\alpha = 1/L$ requires $M_k \geq 120\,\kappa^{3/2}\hat{V}_{G,k}Q_{m}$. In gradient evaluations, the theorem therefore gives a cost of order $\kappa^2 Q_m$ per constant-factor reduction for the accelerated method, against $\kappa Q_m$ for gradient descent, so in our theory the advantage of acceleration is in the number of iterations, or in gradient evaluations when $\nabla f$ is computed exactly or by quadrature. These sample sizes are far from necessary in practice: with $M_k = 2Q_{m_k}$ and $\hat V_{G,k} = 1$ we have $C_G = L^2$, and the theorem only covers $\alpha \leq 1/(3600\kappa^3 L)$, yet our experiments use $\alpha = 1/L$ and still show accelerated behaviour, which suggests the constants are conservative. The factor $1 + \kappa + 4\alpha C_G/\mu$ in Corollary \ref{corol:UQAGDFixed} only adds a lower-order number of iterations (see Appendix \ref{sec:suppAGDCorollary}).

    \subsection{Linear convergence phase}
    As for gradient descent, we combine the fixed-level result with Lemma
    \ref{lemma::LinConvRemainder}. Accelerated methods are not monotone, so instead
    of distances we track the Lyapunov function of Theorem \ref{theorem::AccGradLevel},
    and allow its target to move as the number of basis functions grows. For a level
    $m$ we write $\xi^*_m = (u^*_m, u^*_m)$ and
    \begin{equation*}
    V^{(m)}(\xi) = (\xi - \xi^*_m)^T Q_\alpha (\xi - \xi^*_m)
    + \bar{f}(\mathcal{I}^{-1}(T\xi)) - \bar{f}(x^*_m),
    \end{equation*}
    where $Q_\alpha$ is formed with a constant $\bar{C}_G \geq C_{G,k}$ for all $k \geq K$. We also write $s_k = Q_{m_k}\hat{V}_k/M_k$ for the part of the gradient noise that
    does not vanish at the optimum, and $\rho_a = 1 - \sqrt{\alpha\mu}/3$.
    \begin{lemma} \label{lemma::levelChangeLyap}
    Let $m \leq m'$, let $\xi$ satisfy $T\xi \in l^2_m$, let $\lambda > 0$, and also suppose \newline $\alpha \leq \min\left\{\frac{1}{L}, \frac{\mu^3}{(60\bar{C}_G)^2}\right\}$. Then
    \begin{equation*}
    V^{(m')}(\xi) \leq (1+\lambda)V^{(m)}(\xi)
    + \left[3(1+\lambda^{-1})(L+\mu) + \frac{L}{2}\right]
    ||R_m||_2^2 .
    \end{equation*}
    \end{lemma}
    \begin{proof}
    For the quadratic part, write $\xi - \xi^*_{m'} = (\xi - \xi^*_m)
    + (\xi^*_m - \xi^*_{m'})$ and apply Young’s inequality with $\varepsilon$ (here denoted as $\lambda$) to the quadratic parts of $V^{(m')}(\xi)$, $V^{(m)}(\xi)$. This results in:
    \begin{equation*}
        (\xi - \xi^*_{m'})^T Q_\alpha (\xi - \xi^*_{m'}) \leq (1 + \lambda) (\xi - \xi^*_m)^T Q_\alpha (\xi - \xi^*_m)  +(1  + \lambda^{-1}) (\xi^*_m - \xi^*_{m'})^T Q_\alpha (\xi^*_m - \xi^*_{m'})
    \end{equation*}Both components of $\xi^*_m - \xi^*_{m'}$ equal $\delta \defeq u^*_m - u^*_{m'}$. The vectors $p_\alpha$ and $c$ defining $\tilde{Q}_\alpha$ in Theorem \ref{theorem::AccGradLevel} have entries summing to $\sqrt{\mu/2}$ and $1$ respectively, so
    \begin{equation*}
    (\xi^*_m - \xi^*_{m'})^T Q_\alpha (\xi^*_m - \xi^*_{m'}) = \left(\frac{\mu}{2} + 2\alpha\bar{C}_G\right)||\delta||_2^2 .
    \end{equation*}
    The step size condition gives $60\alpha\bar{C}_G \leq \sqrt{\alpha\mu}\,\mu \leq \mu$, so $\frac{\mu}{2} + 2\alpha\bar{C}_G \leq \frac{8\mu}{15}$, and $||u^*_m - u^*_{m'}||_2 \leq 2\sqrt{\kappa+1}\,||R_m||_2$ by Lemma
    \ref{lemma::LinConvRemainder}. Since $\mu(\kappa+1) = L + \mu$, we can now bound this term:
    \begin{equation*}
        (\xi^*_m - \xi^*_{m'})^T Q_\alpha (\xi^*_m - \xi^*_{m'}) = \left(\frac{\mu}{2} + 2\alpha\bar{C}_G\right) ||\delta||_2^2 \leq \frac{32}{15}(L+\mu)||R_m||_2^2 \leq 3(L+\mu)||R_m||_2^2,
    \end{equation*}and multiplying by the factor $(1+\lambda^{-1})$ from Young's inequality gives the corresponding term in the Lemma. For the remaining part,
    $\bar{f}(T\xi) - \bar{f}(x^*_m) \geq 0$ since $T\xi \in l^2_m$, and
    
    $0 \leq \bar{f}(x^*_m) - \bar{f}(x^*_{m'}) \leq \bar{f}(x^*_m) - \bar{f}(x^*)
    \leq \bar{f}(\mathcal{I}^{-1}(P_m u^*)) - \bar{f}(x^*)
    \leq \frac{L}{2}||R_m||_2^2$. Here the third inequality holds because $x^*_m$ minimises $\bar{f}$ over $L^2_{\pi,m}$, which contains $\mathcal{I}^{-1}(P_m u^*)$, and the last follows from the Lipschitz continuity of the gradient together with $\nabla f(x^*(\theta),\theta) = 0$, since $||P_m u^* - u^*||_2 = ||R_m||_2$.
    \end{proof}
    \begin{theorem} \label{theorem::AccGradOverall}
    Consider Algorithm \ref{algo_accel} with constant $\alpha$ and $\beta$ as in
    Theorem \ref{theorem::AccGradLevel}, and suppose that for all $k \geq K$,
    \begin{equation*}
    \alpha \leq \min\left\{\frac{1}{L}, \frac{\mu^3}{(60\bar{C}_G)^2}\right\}
    \qquad \text{and} \qquad ||R_{m_k}||_2 \leq \varepsilon .
    \end{equation*}
    Let $\bar{s} = \sup_{k \geq K} s_k$ and
    $\tilde{D} = \frac{3(L+\mu)}{1-\rho_a^{1/2}} + \frac{L}{2}
    + 2\alpha\rho_a^{-1/2}\bar{C}_G$. Then for all $N \geq 0$,
    \begin{align*}
    \mathbb{E}\, V^{(m_{K+N})}(\xi_{K+N}) &\leq \rho_a^{N/2}\,
      \mathbb{E}\, V^{(m_K)}(\xi_K)
      + \frac{\tilde{D}\varepsilon^2 + 2\alpha\rho_a^{-1/2}\bar{s}}{1-\rho_a^{1/2}}, \\
    \mathbb{E}||u^{K+N} - u^*||_2^2 &\leq \frac{4}{\mu}\,
      \mathbb{E}\, V^{(m_{K+N})}(\xi_{K+N}) + 2(\kappa+1)\varepsilon^2 .
    \end{align*}
    \end{theorem}
	\begin{proof}
	Fix $k \geq K$ and write $\Phi_k = V^{(m_k)}(\xi^k)$. The iterates satisfy $u^{k-1}, u^k \in l^2_{m_k}$ and $T\xi^{k+1} = u^{k+1} \in l^2_{m_k}$. Lemma \ref{lemma::levelChangeLyap} with $m = m_k$, $m' = m_{k+1}$ and $1 + \lambda = \rho_a^{-1/2}$, so that $1 + \lambda^{-1} = (1-\rho_a^{1/2})^{-1}$, gives
	\begin{equation*}
	\Phi_{k+1} \leq \rho_a^{-1/2}\, V^{(m_k)}(\xi^{k+1}) + \left[\frac{3(L+\mu)}{1-\rho_a^{1/2}} + \frac{L}{2}\right]||R_{m_k}||_2^2 .
	\end{equation*}
	Theorem \ref{theorem::AccGradLevel} is a one-step inequality, so it applies to the step from $\xi^k$ at level $m_k$, with constants $C_k$ and $\bar{C}_G$ in place of $C$ and $C_G$. Conditioning on $\mathcal{F}_{k-1}$, applying it, and then taking total expectations,
	\begin{equation*}
	\mathbb{E}\Phi_{k+1} \leq \rho_a^{1/2}\,\mathbb{E}\Phi_k + 2\alpha\rho_a^{-1/2} C_k + \left[\frac{3(L+\mu)}{1-\rho_a^{1/2}} + \frac{L}{2}\right]||R_{m_k}||_2^2 .
	\end{equation*}
	Since $\frac{2Q_{m_k}\hat{V}_{G,k}}{M_k} = C_{G,k}/L^2$, and $||\nabla f(x^*_{m_k}(\theta),\theta)||_\pi \leq L||R_{m_k}||_2$ by Corollary \ref{corol:gradLipschitzRem}, we have $C_k \leq \bar{C}_G||R_{m_k}||_2^2 + s_k$, and therefore
	\begin{equation*}
	\mathbb{E}\Phi_{k+1} \leq \rho_a^{1/2}\,\mathbb{E}\Phi_k + \tilde{D}\,||R_{m_k}||_2^2 + 2\alpha\rho_a^{-1/2}s_k .
	\end{equation*}
	Iterating from $k = K$ and summing the geometric series gives the first inequality. For the second, $u^{K+N} \in l^2_{m_{K+N}}$, so strong convexity and the first order optimality of $u^*_{m_{K+N}}$ on $l^2_{m_{K+N}}$ give $V^{(m_{K+N})}(\xi^{K+N}) \geq \frac{\mu}{2}||u^{K+N} - u^*_{m_{K+N}}||_2^2$. Finally, we can conclude using $||a+b||_2^2 \leq 2||a||_2^2 + 2||b||_2^2$ and $||u^*_{m_{K+N}} - u^*||_2^2 \leq (\kappa+1)\varepsilon^2$.
	\end{proof}
    The expected Lyapunov function contracts by $\rho_a^{1/2}$ per iteration, so $(\mathbb{E}||u^{K+N}-u^*||_2^2)^{1/2}$ contracts by $\rho_a^{1/4} \leq 1 - \sqrt{\alpha\mu}/12$ per iteration, which is an accelerated rate. The neighbourhood consists of two parts. The first is controlled by the remainder, and so vanishes as the number of basis functions grows. When $\alpha \asymp 1/L$, it is of order $\kappa\varepsilon$ compared with order $\kappa^{3/2}\varepsilon$ for gradient descent in Theorem \ref{theorem::gradConvLinear}. The second is a Monte Carlo part, of order $\sqrt{\alpha}\,\bar{s}/\mu^{3/2}$ in squared error, or $\alpha^{1/4}\bar{s}^{1/2}/\mu^{3/4}$ in distance, which vanishes when $V = 0$.

    \begin{corollary} \label{corol::exactLinearAGD}
    In the setting of Theorem \ref{theorem::AccGradOverall}, suppose instead that
    $||R_{m_k}||_2 \leq A r^k$ and $s_k \leq B r^{2k}$ for all $k \geq K$, for some
    $r \in (0,1)$ with $r^2 \neq \rho_a^{1/2}$, and let
    $G = \tilde{D}A^2 + 2\alpha\rho_a^{-1/2}B$. Then
    \begin{equation*}
    \mathbb{E}\, V^{(m_{K+N})}(\xi_{K+N}) \leq
    \left(\mathbb{E}\, V^{(m_K)}(\xi_K) + \frac{G\, r^{2K}}{|\rho_a^{1/2} - r^2|}\right)
    \max(\rho_a^{1/2}, r^2)^N ,
    \end{equation*}
    and also:
    \begin{align*}
    &\mathbb{E}\|u^{K+N} - u^*\|_2^2 \\
    &\le
    \left( \frac{4}{\mu}\,\mathbb{E}\, V^{(m_K)}(\xi^K)
    + \left[ \frac{4G}{\mu\,|\rho_a^{1/2} - r^2|} + 2(\kappa+1)A^2 \right] r^{2K} \right)
    \max(\rho_a^{1/2}, r^2)^N .
    \end{align*}
    \end{corollary}
    \begin{proof}
	As in the proof of Theorem \ref{theorem::AccGradOverall}, $\mathbb{E}\Phi_{k+1} \leq \rho_a^{1/2}\,\mathbb{E}\Phi_k + G\, r^{2k}$ for all $k \geq K$. Iterating, and using $\sum_{j=0}^{N-1}\rho_a^{(N-1-j)/2} r^{2j} = \frac{\rho_a^{N/2} - r^{2N}}{\rho_a^{1/2} - r^2}$, gives the first inequality. The second follows as in the proof of Theorem \ref{theorem::AccGradOverall}.
	\end{proof}

    We can apply the discussion of Section \ref{ss::totalCost} here also. The lower bound in Equation \eqref{eq::remainderLowerBound} holds for any iterate. The recursion in the proof of Theorem \ref{theorem::AccGradOverall} has the same form as that of Theorem \ref{theorem::gradConvLinear}, with $||R_{m_k}||_2^2$ in place of $||R_{m_k}||_2$ and $\rho_a^{1/2}$ in place of $\rho^{1/2}$.

    We also have that Proposition \ref{prop::tracking} can be used to give $\mathbb{E}||u^k - u^*||_2^2 = O(||R_{m_k}||_2^2)$ once the initial error has been forgotten, when $V = 0$, provided we use an analogue of Equation \eqref{eq::slowRemainder} for $\varepsilon_k^2$ with $\lambda < \rho_a^{-1/2}$.

    As Theorem \ref{theorem::AccGradLevel} requires $M_k$ of order $\kappa^{3/2}Q_{m_k}$ at $\alpha = 1/L$, we do not give an analogue to Corollary \ref{cor:complexity}, as the resulting bound would be worse than the gradient descent bound, as discussed after Corollary \ref{corol:UQAGDFixed}. With exact gradients, $C_G = 0$ and the accelerated method needs a factor of order $\sqrt{\kappa}$ fewer iterations.

	\section{Experiments} \label{sec::Experiments}
	We test both methods on a benchmark quadratic following \cite{crepey2020uncertainty} and on the portfolio problem of Example \ref{exmp::marko}, compare them with naive Monte Carlo and with the methods of \cite{dong2025beyond}, and discuss how to choose $m_k$. The hardware, random seeds and the computation of the reference solutions are described in Appendix \ref{sec:suppExpDetails}. Code to reproduce all figures and tables is available at \url{https://github.com/whipwhipwhipwhip/uq-gradient-descent}.

    \subsection{Benchmark Quadratic}
    
    We take the optimum used in \cite{crepey2020uncertainty}, chosen for the relatively slow decay of its basis coefficients,
	\begin{equation}
	x^*(\theta) = |4/5 + 1/4 \exp(\sin(\theta)) - \cosh(\sin(\theta)^2)|(1+\sin(2\theta)),
	\end{equation}
	and minimise the quadratic
	\begin{equation}
	f(x(\theta),y(\theta), \theta) = \frac{\mu}{2}(x(\theta) - x^*(\theta))^2 + \frac{L}{2}(y(\theta) -x^*(\theta))^2,
	\end{equation}
	which does not depend on $v$, and has the same constants on every level $m$ subspace, with $\kappa = L / \mu$. We use the trigonometric basis, for which $Q_{m_k} = \mathcal{O}(m_k)$, with $\pi$ uniform on $[-\pi, \pi]$. Except in Figure \ref{fig::UQJustification}, where the two methods use different numbers of coefficients and we report the total error, we measure the error $||u^k - P_{m_k}(u^*)||_2^2 = ||u^k - u^*||_2^2 - ||R_{m_k}||_2^2$ on the coefficients in use, which excludes the truncation error.
	
	Since $V = 0$, the Monte Carlo part $\bar\sigma$ of the neighbourhoods in Theorems \ref{theorem::gradConvLinear} and \ref{theorem::AccGradOverall} vanishes. The quadratic is separable, so $u^*_m = P_m(u^*)$, and our error measure isolates the optimisation error, which converges linearly. The total error decays polynomially in $k$ for the polynomial schedules used here, at the rate of the remainder (Proposition \ref{prop::tracking}).
	
    \begin{figure}
        \centering
        \begin{subfigure}{0.32\textwidth}
            \includegraphics[width=\textwidth]{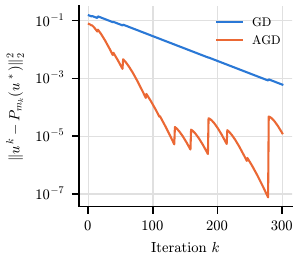}
            \caption{Exact gradients.}
            \label{fig::MultiLevelNoNoise}
        \end{subfigure}
        \hfill
        \begin{subfigure}{0.32\textwidth}
            \includegraphics[width=\textwidth]{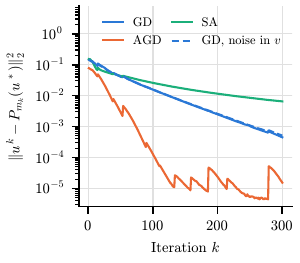}
            \caption{Monte Carlo gradients.}
            \label{fig::NesterovAveraged}
        \end{subfigure}
        \hfill
        \begin{subfigure}{0.32\textwidth}
            \includegraphics[width=\textwidth]{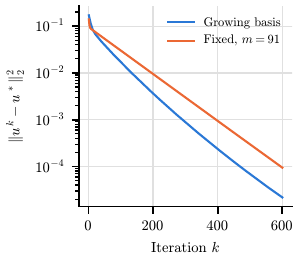}
            \caption{Growing vs fixed basis.}
            \label{fig::UQJustification}
        \end{subfigure}
        \caption{The benchmark quadratic, averaged over $200$ trials. (a), (b) Gradient descent and accelerated gradient descent with a growing basis, with exact and Monte Carlo gradients; (b) also shows the stochastic approximation method of \cite{crepey2020uncertainty}, and the dashed line adds noise through $v$. (c) Gradient descent with a growing basis against a fixed level $m = 91$, in total error $||u^k - u^*||_2^2$.}
        \label{fig:figures}
    \end{figure}		
    
	In Figure \ref{fig::MultiLevelNoNoise} the algorithm is given $D_{m_k}$ exactly, with $L = 200$, $\mu = 1$, $m_k = \lfloor \sqrt{k+10}\rfloor+2$, zero initial conditions and $300$ iterations. We use the maximal step sizes of Theorems \ref{theorem::GDFixedUQ} and \ref{theorem::AccGradLevel} with exact gradients ($C = 0$, and $C_G = 1$ and $0$ respectively), namely $\gamma = \frac{2}{\mu + L}$ and $\alpha = 1/L$, $\beta = \frac{1 - \sqrt{\alpha \mu}}{1 + \sqrt{\alpha \mu}}$, which remain constant as the basis grows. The spikes in the accelerated curve occur when a new basis function is added, since its coefficient is initialised to zero, and have approximately the same magnitude for both methods.
	
    In Figure \ref{fig::NesterovAveraged} the gradient is estimated by Monte Carlo with $M_k = 500$, with all other parameters as before apart from the step size. Gradient descent uses the maximal step size $\frac{2}{(\mu+L)C_G}$ with $C_G = 1 + 2Q_{m_k}/M_k$ (Corollary \ref{corol:Noisemodel}, as $V_G = 1$), which decreases as $m_k$ grows. The accelerated method keeps its noiseless step sizes, since Theorem \ref{theorem::AccGradLevel} would require far more samples ($M_k \geq 120\kappa^{3/2}Q_{m_k}$ at $\alpha = 1/L$). We also include the stochastic approximation method of \cite{crepey2020uncertainty}, with the same estimator and decaying step sizes (see Appendix \ref{sec:suppExpQuadratic}). The results are qualitatively similar, except that noise affects the accelerated method at lower errors, as expected given its sample sizes and the greater noise sensitivity of accelerated methods \cite{aybat2020robust}. Without the factor $1/C_G$ in its step size, gradient descent blew up, while the accelerated method still converged, which suggests the step size conditions are not tight.

    The dashed line in Figure \ref{fig::NesterovAveraged} adds noise through $v \sim \mathcal{U}[-1,1]$,
	\begin{equation}
	F(x(\theta),y(\theta), \theta,v) = \frac{\mu}{2}(x(\theta) - x^*(\theta))^2 + \frac{L}{2}(y(\theta) -x^*(\theta))^2 + v\left(x(\theta) + y(\theta)\right),
	\end{equation}
	so that $\mathbb{E}_V F = f$, $\nabla F = \nabla f + v$, $V = 2/3$ and $V_G = 1$. The step size is unchanged, as it depends only on $\hat{V}_{G,k} = 1$, and $V$ only affects the Monte Carlo part $\bar\sigma$ of the neighbourhood (Theorem \ref{theorem::gradConvLinear}). The two curves are almost indistinguishable over 300 iterations, since the error is still well above that floor, though the noisy run begins to lag slightly by $k = 300$.
	
	\subsubsection{Impact of basis functions on noise} \label{sssec:varRed}
	We now demonstrate the drawback of a large fixed level discussed in Example \ref{exmp::VarianceExplain}. In Figure \ref{fig::UQJustification} we run $600$ iterations on the same quadratic with $\mu = 1$, $L = 200$ and $m_k = \lfloor (k+10)^{0.7} \rfloor + 2$, against a fixed level $m = 91 = m_{600}$, the largest level reached by the growing basis, where the truncation error is approximately $7.1 \times 10^{-7}$. Both use $250$ Monte Carlo samples, so in the fixed level case $C_G = 1+ 2Q_{91}/M_k \leq 1.728$. The growing basis wins because, with $M_k$ fixed, its maximal step size $2/((\mu+L)C_G)$ is larger while $Q_{m_k}$ is small. Each of its iterations is also cheaper, and at equal computational time it reaches an error about $60$ times smaller than the fixed level algorithm (timings are given in Appendix \ref{sec:suppExpQuadratic}).
	
    \subsection{Portfolio Optimisation} \label{ss:portfolioOpt}
    We return to Example \ref{exmp::marko} with $n = 10$ assets, whose covariance follows a two-factor model,
    \begin{equation}
          \Sigma(\theta) = \operatorname{diag}(D) + \sum_{j=1}^{2} s_j e^{\gamma\theta_j} b_j b_j^\top,
          \qquad \theta \sim U[-1,1]^2,
    \end{equation}
    so each coordinate of $\theta$ scales the variance of one risk factor by a factor of up to $e^{\pm \gamma}$, and minimise
    \begin{equation}
          f(w(\theta),\theta) = -w(\theta)^\top r + \frac{\lambda}{2} w(\theta)^\top \Sigma(\theta) w(\theta)
          + \frac{\tau}{2}\|w(\theta) - w_0\|_2^2 + \frac{\eta}{2}\bigl(w(\theta)^\top e - 1\bigr)^2,
    \end{equation}
    with the parameter values and the random draw of $r$, $D$ and $b_j$ given in Appendix \ref{sec:suppExpPortfolio}. The budget constraint is imposed only through the penalty, so the optimal portfolio need not be fully invested; as discussed in the same section, the remainder can be read as a holding in cash.

    The extreme eigenvalues of the Hessian occur at the corners of $[-1, 1]^2$, giving $\mu \approx 0.27$, $L \approx 22.5$ and $\kappa \approx 84$. Unlike the benchmark quadratic, $\Sigma(\theta)$ couples the coefficients of $w$, so $P_m(u^*) \neq u^*_m$ in general. To model the fact that returns are only estimated, we take $F = f - w^\top v$ with $v \sim U[-0.05, 0.05]^n$, so that $V = n(0.05)^2/3$ and $V_G = 1$.

    We use the tensor-product Legendre basis, ordered by total degree with ties broken lexicographically. The total degree is $t_k = \lceil k/25 \rceil$, $m_k$ is the number of coefficients with total degree $\leq t_k$, capped at $t = 6$ ($m = 28$) in this subsection, and $M_k = 2 Q_{m_k}$, $P = 1$.

    \begin{figure}
        \centering
        \includegraphics[width=0.5\linewidth]{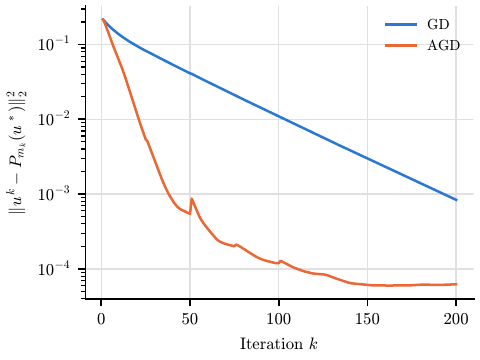}
		\caption{Convergence on Markowitz problem for GD, AGD, averaged over $20$ trials. Note that the accelerated method has already converged to a noise floor.} 
		\label{fig::MarkowitzConverge}
    \end{figure}
    
    Both methods converge to a neighbourhood of the true solution (Figure \ref{fig::MarkowitzConverge}); as expected, the accelerated method converges faster, but with a higher noise floor.

    \begin{figure}
        \centering
        \includegraphics[width=0.8\linewidth]{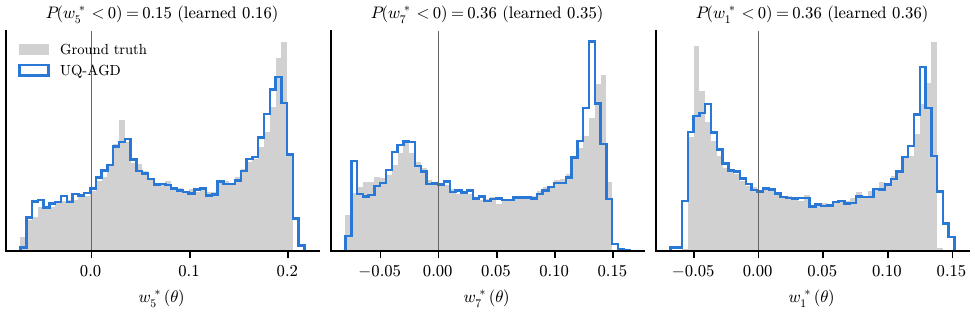}
		\caption{A few computed components of $w^*(\theta)$. The components were chosen to have interesting non-trivial probability of short selling $w^* < 0$.} 
		\label{fig::MarkowitzUQ}
    \end{figure}

    A single run also recovers the distribution of the optimal weights. Figure \ref{fig::MarkowitzUQ} shows it for three assets, against a ground truth from naive Monte Carlo on $20,000$ samples of $\theta$; both are computed on the same samples of $\theta$, and the Monte Carlo error of their differences is discussed in Appendix \ref{sec:suppShortSale}. Table \ref{tab:asset_statistics} gives summary statistics, including the probability of a short sale, $w_i < 0$, which matters in applications since short selling often incurs borrowing costs.

    \begin{table}[htbp]
        \centering
        \caption{Comparison of true and learned statistics for individual assets in the Markowitz problem.}
        \label{tab:asset_statistics}
        \small
        \setlength{\tabcolsep}{3pt}
        \begin{tabular}{ccccccccccc}
            \toprule
            Asset & \multicolumn{2}{c}{Mean} & \multicolumn{2}{c}{SD} & \multicolumn{2}{c}{5\%} & \multicolumn{2}{c}{95\%} & \multicolumn{2}{c}{$P(w_i<0)$} \\
            & True & Learned & True & Learned & True & Learned & True & Learned & True & Learned \\
            \midrule
            1  & 0.042 & 0.042 & 0.066 & 0.067 & -0.049 & -0.050 & 0.134 & 0.135 & 0.361 & 0.363 \\
            2  & 0.064 & 0.065 & 0.031 & 0.031 & 0.022 & 0.023 & 0.125 & 0.128 & 0.000 & 0.000 \\
            3  & -0.165 & -0.165 & 0.086 & 0.086 & -0.293 & -0.293 & -0.045 & -0.042 & 1.000 & 1.000 \\
            4  & 0.002 & 0.002 & 0.020 & 0.021 & -0.026 & -0.027 & 0.034 & 0.038 & 0.484 & 0.485 \\
            5  & 0.090 & 0.089 & 0.079 & 0.080 & -0.041 & -0.043 & 0.196 & 0.197 & 0.151 & 0.158 \\
            6  & 0.286 & 0.287 & 0.008 & 0.009 & 0.275 & 0.273 & 0.299 & 0.300 & 0.000 & 0.000 \\
            7  & 0.043 & 0.043 & 0.071 & 0.070 & -0.064 & -0.064 & 0.141 & 0.139 & 0.358 & 0.354 \\
            8  & 0.128 & 0.128 & 0.026 & 0.026 & 0.078 & 0.077 & 0.159 & 0.158 & 0.000 & 0.000 \\
            9  & 0.109 & 0.109 & 0.008 & 0.008 & 0.095 & 0.096 & 0.123 & 0.123 & 0.000 & 0.000 \\
            10 & 0.228 & 0.227 & 0.029 & 0.029 & 0.182 & 0.181 & 0.268 & 0.269 & 0.000 & 0.000 \\
            \bottomrule
        \end{tabular}
    \end{table}
    Convergence in $L^2_\pi$ does not control threshold probabilities such as $P(w_i < 0)$ in general. If $w^*_i$ has a density bounded by $c$ near $0$, however, then \newline $\left|\Pr(\hat{w}_i < 0) - \Pr(w^*_i < 0)\right| \leq 3\left(c\,\|\hat{w}_i - w^*_i\|_{L^2_\pi}\right)^{2/3}$ (see Appendix \ref{sec:suppShortSale}).

    \subsection{Comparison to other methods}

    We compare methods by the number of evaluations of $\nabla F$ needed to reach a given $L^2_\pi$ error $\|x - x^*\|_\pi^2 = \|u - u^*\|_2^2$ on the portfolio problem with $V = 0$, without assuming that a suitable number of basis functions is known in advance. Learning the Solution Path (LSP), the primary method of \cite{dong2025beyond}, is stochastic gradient descent with one sample of $\theta$ per step and a fixed $m$; we run it with $m = 10, 28, 66$, that is total degree $3, 6, 10$. Adaptively Learning the Solution Path (ALSP) appends a new basis function whenever the objective has decreased by less than $1\%$. The naive method samples $\theta_j \sim \pi$, solves each problem by gradient descent, and estimates the coefficients by Monte Carlo.

    Our methods use the same schedule as in Section \ref{ss:portfolioOpt}, but capped at total degree 10 ($m = 66$), the largest level used by any method here. For the naive method we report, for each sample size, its best truncation level in hindsight, which favours this baseline. Step sizes, the inner solver of the naive method, and the evaluation of the stall rule are given in Appendix \ref{sec:suppExpPortfolio}.

    \begin{figure}
        \centering
        \includegraphics[width=0.9\linewidth]{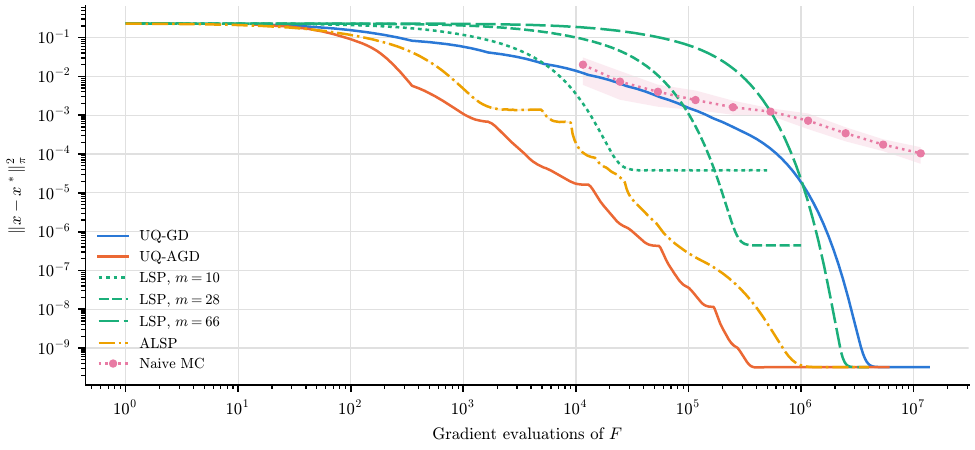}
		\caption{Markowitz with $V=0$ for different methods. Fixed-size LSP either stalls at its truncation error, or converges slowly because of large $Q_m$. ALSP uses its best setting from a sweep. Averaged over $20$ trials. Shaded bands show $10$th and $90$th percentiles.} 
		\label{fig::methodComparison}
    \end{figure}

    Figure \ref{fig::methodComparison} plots the error against the number of gradient evaluations. Unless $m$ is set very high, LSP stalls prematurely, but a suitable $m$ is hard to know in advance, and too large an $m$ converges slowly.
    
    ALSP increases $m$ when convergence slows sufficiently, whereas our methods so far add a total degree every $25$ iterations. The accelerated method is the fastest at every accuracy, but UQ-GD is the slowest of the methods that converge. This is due to the schedule, which suits the accelerated method, whose rate depends on $\sqrt{\kappa}$, more than gradient descent, whose rate depends on $\kappa \approx 84$. A slower schedule (Table \ref{tab:test}) reduces the gradient evaluations of UQ-GD by almost an order of magnitude, to within a factor of two to three of ALSP.

    \subsubsection{Scaling with $\kappa$}
    Finally, we use the class-reweighted logistic regression of \cite{dong2025beyond}, on their 1000-row subsample of the LSAC National Longitudinal Bar Passage Study \cite{wightman1998}, which has 45 standardised features and 956 positive and 44 negative labels. With $x = (w, b) \in \mathbb{R}^{46}$ the weights and intercept, and $\ell_i$ the logistic loss of row $i$,
    \begin{equation}\label{eq:logistic}
      f(x, \lambda) = (1 - \lambda) \operatorname*{mean}_{i:\, y_i = 1} \ell_i(x)
      + \lambda \operatorname*{mean}_{i:\, y_i = 0} \ell_i(x) + \frac{\rho}{2}\|x\|_2^2,
      \qquad \lambda \sim U[0, 1].
    \end{equation}
    Due to the ridge, $f$ is $\rho$-strongly convex. 
    
    Its gradient is Lipschitz with $L = \rho + c$, where $c \approx 3.67$ (see Appendix \ref{sec:suppExpLogistic}). Their setting used $\rho = 1/4$, giving $\kappa \approx 16$. We also use $\rho = 1/16$, $1/64$ and $1/256$, giving
    $\kappa \approx 60$, $236$ and $940$.

    We use the Legendre basis in $\theta = 2\lambda - 1$ and measure the relative error $\|u - u^*\|_2^2 / \|u^*\|_2^2$. Both methods grow the basis up to the same cap $m_{\max}$, set from $u^*$ only so that all runs stop at a common achievable accuracy; it is never used to choose when to add functions.

    We compare ALSP with the accelerated method under two growth rules: ALSP's own stall rule, and a fixed schedule that adds one basis function every $\lceil 3\sqrt{\kappa}\rceil$ iterations. In both cases the accelerated method uses $\alpha = 1/L$, starts from $m = 1$, and estimates the gradient with $M_k = 2Q_{m_k}$ samples, where $Q_m = m^2$ for the Legendre basis. ALSP, like our stall rule, starts from $m = 1$. The check intervals of both methods were tuned for each $\kappa$ with the same selection rule, the fewest gradient evaluations to reach an accuracy of $10^{-5}$ (details in Appendix \ref{sec:suppExpLogistic}). All curves are means over 20 trials, except 5 at $\kappa \approx 940$, where each ALSP run is long.

    \begin{figure}[t]
      \centering
      \includegraphics[width=\textwidth]{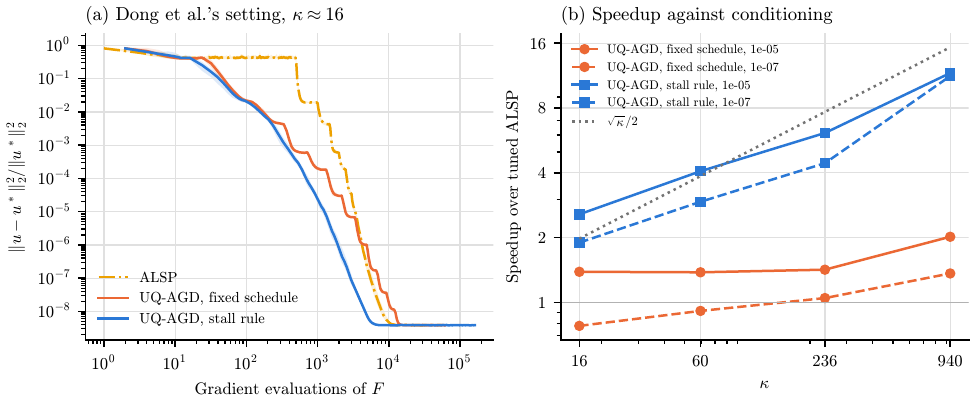}
      \caption{The class-reweighted logistic regression of \cite{dong2025beyond}.
        (a)~Relative error against gradient evaluations at their setting, $\kappa \approx 16$;
        means over 20 trials. (b)~Speedup of the
        accelerated method over the tuned ALSP against $\kappa$, at relative errors $10^{-5}$ (solid) and $10^{-7}$ (dashed), with $\sqrt{\kappa}/2$ for reference; 5 trials at
        $\kappa \approx 940$.}
      \label{fig:logistic}
    \end{figure}

    With the same growth rule, the speedup of the accelerated method over ALSP grows with $\kappa$ (Figure~\ref{fig:logistic}(b), tabulated in Table~\ref{tab:logistic} in Appendix \ref{sec:suppExpLogistic}).
    Fitted over the four values of $\kappa$, at different relative errors, the speedup grows roughly like $\kappa^{0.4}$, close to the $\sqrt{\kappa}$ ratio between the
    iteration complexities of accelerated and plain gradient descent. At relative errors $10^{-3}$, $10^{-5}$ and $10^{-7}$, it rises from 4.2, 2.6 and 1.9 at their setting ($\kappa \approx 16$) to 23.4, 11.6 and 11.3 at $\kappa \approx 940$, somewhat below $\sqrt{\kappa}/2 = 15.3$ at high accuracy. That ratio compares rates on a fixed basis, whereas part of the total cost comes from adding functions: the method cannot reach a given accuracy until the basis is large enough, and after each new function the error in its coefficient must be brought down before linear convergence resumes.
    
    The fixed schedule does not benefit in the same way: it is only 1.4--2.0 times faster than ALSP at $10^{-5}$, and at most 1.4 times faster at $10^{-7}$. Its rate was carried over from the portfolio problem, where each growth step added a whole total degree; here each step adds a single function, so the basis grows three to ten times more slowly, and the method spends most of its iterations waiting for the next function (the steps in Figure~\ref{fig:logistic}(a)). A fixed schedule only performs well when its rate suits the problem.

    \subsection{Choosing $m_k$} \label{ss:choosingMk}
    Our analysis allows any $m_k$ for which the remainder decays to zero, but the choice affects the cost.

\begin{table}[htbp]
    \centering
    \caption{Gradient evaluations needed for given convergence thresholds for Markowitz problem with $V = 0$. Results for UQ-GD and UQ-AGD with different total degree increase rates.}
    \label{tab:test}
    \begin{tabular}{ccccc}
        \toprule
        Method & Degree increase rate & $10^{-4}$ & $10^{-6}$ & $10^{-8}$ \\
        \midrule
        GD  & every 25  & $6 \times 10^{5}$ & $2 \times 10^{6}$ & $3 \times 10^{6}$ \\
        GD  & every 75  & $4 \times 10^{4}$ & $2 \times 10^{5}$ & $8 \times 10^{5}$ \\
        GD  & every 200 & $2 \times 10^{4}$ & $2 \times 10^{5}$ & $1 \times 10^{6}$ \\
        AGD & every 25  & $4 \times 10^{3}$ & $3 \times 10^{4}$ & $2 \times 10^{5}$ \\
        AGD & every 75  & $7 \times 10^{3}$ & $9 \times 10^{4}$ & $5 \times 10^{5}$ \\
        AGD & every 200 & $2 \times 10^{4}$ & $2 \times 10^{5}$ & $1 \times 10^{6}$ \\
        \bottomrule
    \end{tabular}
\end{table}

    Table \ref{tab:test} shows the cost on the portfolio problem ($\kappa \approx 84$) when the total degree increases as $t_k = \lceil k/c \rceil$ for varying $c$. The best rate depends on the method, and a natural conjecture is that a good rate adds functions approximately every $\kappa$ iterations for gradient descent, and every $\sqrt{\kappa}$ iterations for the accelerated method.

    To be more exact, it seems the correct way to add basis functions is to wait for when the truncation error roughly equals the optimisation error, which is consistent with Proposition \ref{prop::tracking} and Corollary \ref{cor:complexity}. ALSP from \cite{dong2025beyond} is an initial heuristic in that direction. Analysing instances of such rules and their error rates is a natural direction for future work.
    
	\section{Conclusion} \label{sec::Conclusion}

	We introduced gradient descent and accelerated gradient descent counterparts of the stochastic approximation method of \cite{crepey2020uncertainty} for uncertainty quantification. Whereas estimating the statistics of $x^*(\theta)$ naively requires a separate run for every sample of $\theta$, $O(\delta^{-2})$ runs for even the mean to accuracy $\delta$, our methods compute all the basis coefficients, and hence the whole function, within a single run. Each method converges linearly to a neighbourhood of the solution whose radius is controlled by the truncation error and, when $V > 0$, by the sample sizes. When $V = 0$, the error decays at the rate of the truncation error, the fastest possible with the same basis functions, and we bound the total number of gradient evaluations needed to reach a given accuracy. Compared with \cite{dong2025beyond}, we also allow acceleration and noisy gradients, and with the same ALSP growth rule our accelerated method needed fewer gradient evaluations, by a factor scaling approximately like $\kappa^{0.4}$ on their logistic regression benchmark. ALSP is a heuristic that adapts $m_k$ to the problem; analysing such rules, with the aim of minimising the total number of gradient evaluations, is interesting future work.

    \section*{Acknowledgments}
    The authors assume responsibility for all content. Opus 5.5 was used to revise this manuscript, initially written prior to LLM availability, and a novelty search with it uncovered the paper \cite{dong2025beyond}. We used Opus to check for typos, notational and presentation issues throughout. It made suggestions in the review process, and aided for making the paper more concise to meet page limits. We used it to verify written proofs, and generated code to run (human-chosen) experiments in Section \ref{sec::Experiments}. It found a more direct version of the Theorem \ref{theorem::AccGradOverall} proof, and made the suggestion to add theorem for total cost, as in Section \ref{ss::totalCost}, which it aided in hypothesising. All AI suggestions and contributions were human-vetted before inclusion.

    \bibliographystyle{siamplain}
    \bibliography{v2026/submission_2/uq_revision}
    \newpage

\appendix
\section{Strong Convexity Properties} \label{sec:suppConvexity}
In this appendix, we give many typical properties of strongly convex functions and show that they still hold under our definition of strong convexity. We recall again that definition:
\begin{definition} \label{appendixdef1}
Consider a function $f: L^2_\pi \times \mathbb{R}^d \rightarrow \mathbb{R}$. We say that $f$ is strongly convex with parameter $\mu$ if:
\begin{equation}
    \langle \nabla f(x_1(\theta),\theta)- \nabla f(x_2(\theta),\theta), x_1(\theta) - x_2(\theta) \rangle_\pi \geq \mu||x_1(\theta) - x_2(\theta)||_\pi^2 
\end{equation}
for all $x_1(\cdot), x_2(\cdot) \in L^2_\pi$.
\end{definition}
Similarly, we can give an equivalent definition for convexity:
\begin{definition} \label{appendixdef2}
Consider a function $f: L^2_\pi \times \mathbb{R}^d \rightarrow \mathbb{R}$. We say that $f$ is convex if:
\begin{equation}
    \langle \nabla f(x_1(\theta),\theta)- \nabla f(x_2(\theta),\theta), x_1(\theta) - x_2(\theta) \rangle_\pi \geq 0
\end{equation}
for all $x_1(\cdot), x_2(\cdot) \in L^2_\pi$.
\end{definition}

We also write $\bar{f}(x) \defeq \int_{\Theta} f(x(\theta),\theta)\,\pi(d\theta)$ for the mean of $f$ at $x$, as in Section \ref{sec::AGD}. For a fixed $\theta$, $\langle \cdot, \cdot \rangle_2$ and $||\cdot||_2$ denote the Euclidean inner product and norm on $\mathbb{R}^q$, and we write $\nabla f(x)$ for $\nabla f(x(\cdot),\cdot)$. By Parseval's identity, each inequality below can equivalently be written in terms of the basis coefficient vectors $\mathcal{I}(x)$ and $D f(x)$, with $\langle \cdot, \cdot \rangle_2$ on $l^2$ in place of $\langle \cdot, \cdot \rangle_\pi$; this is the form used in Lemma \ref{lemma::LinConvRemainder}. All results also hold with $L^2_\pi$ replaced by a level $m$ subspace $L^2_{\pi,m}$ and $\nabla f(x)$ replaced by $P_m \nabla f(x)$, since every point and direction appearing in the proofs then lies in $L^2_{\pi,m}$.

\begin{lemma} \label{appendix::descent}
Suppose that $\langle \nabla f(x) - \nabla f(y), x - y \rangle_\pi \leq L||x - y||_\pi^2$ for all $x, y \in L^2_\pi$, which holds in particular if $\nabla f$ is Lipschitz continuous with parameter $L$. Then for all $x, y \in L^2_\pi$,
\begin{equation*}
\bar{f}(y) \leq \bar{f}(x) + \langle \nabla f(x), y - x \rangle_\pi + \frac{L}{2}||y - x||_\pi^2 .
\end{equation*}
\end{lemma}
\begin{proof}
Fix $x, y$, write $d = y - x$, and for each $\theta$ let $g_\theta(t) = f(x(\theta) + t\,d(\theta), \theta)$, so that $g_\theta'(t) = \langle \nabla f(x(\theta) + t\,d(\theta),\theta), d(\theta) \rangle_2$. By the fundamental theorem of calculus, $f(y(\theta),\theta) = f(x(\theta),\theta) + \int_0^1 g_\theta'(t)\,dt$ for each $\theta$. Integrating over $\theta$, and exchanging the order of integration by Fubini's theorem, which applies since $|g_\theta'(t)|$ is integrable with $\int_0^1\int_\Theta |g_\theta'(t)|\,\pi(d\theta)\,dt \leq \int_0^1 ||\nabla f(x + td)||_\pi ||d||_\pi\, dt < \infty$, we find
\begin{equation} \label{appendix::integralForm}
\bar{f}(y) - \bar{f}(x) - \langle \nabla f(x), d \rangle_\pi = \int_0^1 \langle \nabla f(x + td) - \nabla f(x), d \rangle_\pi \, dt .
\end{equation}
For $t > 0$, the hypothesis applied to the points $x + td$ and $x$ gives $\langle \nabla f(x + td) - \nabla f(x), d \rangle_\pi \leq tL||d||_\pi^2$, and integrating over $t$ gives the result.
\end{proof}

\begin{lemma} \label{appendix::lowerBound}
Let $f$ be strongly convex with parameter $\mu \geq 0$, where $\mu = 0$ corresponds to convexity. Then for all $x, y \in L^2_\pi$,
\begin{equation*}
\bar{f}(y) \geq \bar{f}(x) + \langle \nabla f(x), y - x \rangle_\pi + \frac{\mu}{2}||y - x||_\pi^2 .
\end{equation*}
In particular, if $\nabla f(x^*) = 0$ then $x^*$ minimises $\bar{f}$, and if moreover $\mu > 0$ then $\bar{f}(x) - \bar{f}(x^*) \leq \frac{1}{2\mu}||\nabla f(x)||_\pi^2$ for all $x \in L^2_\pi$.
\end{lemma}
\begin{proof}
In Equation \eqref{appendix::integralForm}, with $d = y - x$, the definition of strong convexity bounds the integrand below by $t\mu||d||_\pi^2$, which gives the inequality. Taking $x = x^*$ with $\nabla f(x^*) = 0$ shows that $x^*$ minimises $\bar{f}$. For the last claim, minimising the right hand side over $y$ gives $\bar{f}(y) \geq \bar{f}(x) - \frac{1}{2\mu}||\nabla f(x)||_\pi^2$ for all $y$, and we take $y = x^*$.
\end{proof}

\begin{lemma} \label{appendixConvexLemma}
Suppose $\nabla f$ satisfies the hypothesis of Lemma \ref{appendix::descent}, and $x^*$ minimises $\bar{f}$. Then for all $x \in L^2_\pi$,
\begin{equation*}
\bar{f}(x) - \bar{f}(x^*) \geq \frac{1}{2L}||\nabla f(x)||_\pi^2 .
\end{equation*}
\end{lemma}
\begin{proof}
Taking $y = x - \frac{1}{L}\nabla f(x)$ in Lemma \ref{appendix::descent} gives $\bar{f}(x^*) \leq \bar{f}(y) \leq \bar{f}(x) - \frac{1}{L}||\nabla f(x)||_\pi^2 + \frac{1}{2L}||\nabla f(x)||_\pi^2$, and rearranging gives the result.
\end{proof}

From this Lemma, we can prove our first main result, an adaptation of \textit{co-coercivity} of the gradient:
\begin{theorem} \label{appendix::cocoercivity}
Let $f$ be convex, and suppose $\nabla f$ satisfies the hypothesis of Lemma \ref{appendix::descent}. Then for all $x, y \in L^2_\pi$,
\begin{equation*}
\langle \nabla f(x) - \nabla f(y), x - y \rangle_\pi \geq \frac{1}{L} ||\nabla f(x) - \nabla f(y)||_\pi^2 .
\end{equation*}
\end{theorem}
\begin{proof}
Fix $x, y$ and define $f_x(z,\theta) = f(z,\theta) - \langle \nabla f(x(\theta),\theta), z \rangle_2$, so that $\nabla f_x(z) = \nabla f(z) - \nabla f(x)$. Then $f_x$ is convex, $\nabla f_x$ satisfies the hypothesis of Lemma \ref{appendix::descent}, and $\nabla f_x(x) = 0$, so $x$ minimises $\bar{f}_x$ by Lemma \ref{appendix::lowerBound}. Lemma \ref{appendixConvexLemma} applied to $f_x$ at the point $y$ gives
\begin{equation*}
\bar{f}(y) - \bar{f}(x) - \langle \nabla f(x), y - x \rangle_\pi = \bar{f}_x(y) - \bar{f}_x(x) \geq \frac{1}{2L}||\nabla f(y) - \nabla f(x)||_\pi^2 .
\end{equation*}
Exchanging the roles of $x$ and $y$ gives $\bar{f}(x) - \bar{f}(y) - \langle \nabla f(y), x - y \rangle_\pi \geq \frac{1}{2L}||\nabla f(x) - \nabla f(y)||_\pi^2$, and adding the two inequalities gives the result.
\end{proof}

We can also, as in the usual case, extend co-coercivity to strongly convex functions:
\begin{theorem} \label{appendix::strongCocoercivity}
Suppose $f$ is strongly convex with parameter $\mu$, and $\nabla f$ is Lipschitz continuous with parameter $L$. Then for all $x, y \in L^2_\pi$,
\begin{equation*}
(L+\mu)\langle \nabla f(x) - \nabla f(y), x - y \rangle_\pi \geq \mu L ||x - y||_\pi^2 + ||\nabla f(x) - \nabla f(y)||_\pi^2 .
\end{equation*}
\end{theorem}
\begin{proof}
Write $g = \nabla f(x) - \nabla f(y)$ and $d = x - y$. If $\mu = L$, then $\langle g, d \rangle_\pi \geq L||d||_\pi^2$ and $||g||_\pi \leq L||d||_\pi$ give $2L\langle g, d \rangle_\pi \geq L^2||d||_\pi^2 + ||g||_\pi^2$, which is the claim. Otherwise $\mu < L$, and we let $h(z,\theta) = f(z,\theta) - \frac{\mu}{2}||z||_2^2$, so that $\nabla h(z) = \nabla f(z) - \mu z$. By the definition of strong convexity, $h$ is convex, and by the Cauchy--Schwarz inequality and Lipschitz continuity, $\langle \nabla h(z_1) - \nabla h(z_2), z_1 - z_2 \rangle_\pi \leq (L - \mu)||z_1 - z_2||_\pi^2$. So Theorem \ref{appendix::cocoercivity} applies to $h$ with $L - \mu$ in place of $L$, and gives
\begin{equation*}
(L - \mu)\left(\langle g, d \rangle_\pi - \mu||d||_\pi^2\right) \geq ||g - \mu d||_\pi^2 = ||g||_\pi^2 - 2\mu\langle g, d \rangle_\pi + \mu^2||d||_\pi^2 ,
\end{equation*}
which rearranges to the claim.
\end{proof}

\section{Proofs for Sections \ref{section::AlgoIntro}--\ref{sec::GD}} \label{sec:suppProofs}
In this section we give the proofs of Theorem \ref{theorem::IntegralError}, Theorem \ref{theorem::GDFixedUQ} and Corollary \ref{cor:complexity}.

\begin{proof}[Proof of Theorem \ref{theorem::IntegralError}]
		First consider just $e_{i,k}$. We split it into two parts as follows:
		\begin{equation*}
		    e_{i,k} = a_{i,k} + b_{i,k}
		\end{equation*}
		where we have:
		\begin{align*}
		    a_{i,k} &= \frac{1}{M_k}\sum_{j=1}^{M_k}\left(  \frac{1}{P_j} \sum_{p = 1}^{P_j} \nabla F(x^k(\theta_j),\theta_j,v_{j,p}) B_i(\theta_j) - \nabla f(x^k(\theta_j), \theta_j) B_i(\theta_j) \right), \nonumber \\
		    b_{i,k} &= \left(\frac{1}{M_k}\sum_{j=1}^{M_k} \nabla f(x^k(\theta_j), \theta_j) B_i(\theta_j)\right) - \int_{\Theta} \nabla f(x^k(\theta),\theta) B_i(\theta) \pi(d\theta),
		\end{align*}
		where $x^k$ is the point found after the first $k-1$ steps from the filtration $\mathcal{F}_{k-1}$. We now seek to bound this error where we also fix the $\theta_j$ used in the $k$th step, so the only randomness is in the $v$ variables used. We label the set of $\theta$ used as $\Theta_k$. Expanding the definition of the expectation of the square, we find:
		\begin{align} \label{lem::ErrorEquationSum}
		    \mathbb{E}(e_{i,k}^2 | \mathcal{F}_{k-1}, \Theta_k)& \nonumber \\
		    &= \mathbb{E}\left( \left(a_{i,k} + b_{i,k}\right)^2 | \mathcal{F}_{k-1}, \Theta_k\right) \nonumber \\
		    &= \mathbb{E}\left( a_{i,k}^2 | \mathcal{F}_{k-1}, \Theta_k\right) + \mathbb{E}\left( b_{i,k}^2 | \mathcal{F}_{k-1}, \Theta_k\right),
		\end{align}
		where we have used the fact that $b_{i,k}$ is constant under this expectation, and \linebreak $\mathbb{E}\left( a_{i,k} | \mathcal{F}_{k-1}, \Theta_k\right) = 0$ to get rid of the cross term. We now bound each of these terms in turn, then take expectations over all $\Theta_k, \mathcal{F}_{k-1}$.
		
		Firstly, note that for fixed $\theta_j$, the samples $v_{j,1}, \ldots, v_{j,{P_j}}$ appearing in $a_{i,k}$ are i.i.d.\ with mean $\nabla f(x^k(\theta_j), \theta_j)$, so their average has variance reduced by a factor $1/P_j$ relative to a single sample:
		\begin{align*}
		    &\mathbb{E}\left( a_{i,k}^2 | \mathcal{F}_{k-1}, \Theta_k\right)  \nonumber \\
            &= \frac{1}{M_k^2} \sum_{j=1}^{M_k} \frac{1}{P_j} \int_{V}\left( \nabla F(x^k(\theta_j),\theta_j,v)  - \nabla f(x^k(\theta_j), \theta_j)  \right)^2 \nu(\theta_j, dv) B_i(\theta_j)^2.
		\end{align*}
		and now taking expectation over $\Theta_k$, we find:
		\begin{align*}
		    &\mathbb{E}\left( a_{i,k}^2 | \mathcal{F}_{k-1} \right)  \\
            &= \left(\frac{1}{M_k^2}\sum_{j=1}^{M_k}\frac{1}{P_j}\right) \int_{\Theta} \int_{V}\left( \nabla F(x^k(\theta),\theta,v)  - \nabla f(x^k(\theta), \theta)  \right)^2 \nu(\theta, dv) B_i(\theta)^2 \pi(d\theta)
		\end{align*}
		Recalling the definition of $\bar{P}_k$, so that $\frac{1}{M_k^2}\sum_{j=1}^{M_k}\frac{1}{P_j} = \frac{1}{M_k\bar{P}_k}$, together with the definition of $Q_{m_k}$, and summing over $i$, we can say:
		\begin{equation*}
		    \mathbb{E}\left( ||a_{k}||_2^2 | \mathcal{F}_{k-1} \right) \leq \frac{Q_{m_k}}{M_k\bar{P}_k} \int_{\Theta} \int_{V}\left( \nabla F(x^k(\theta),\theta,v)  - \nabla f(x^k(\theta), \theta)  \right)^2 \nu(\theta, dv) \pi(d\theta).
		\end{equation*}
		We now just need to bound the inner term. By expanding the square inside and using the bound on the second moment of $\nabla F$, we find:
		\begin{equation*}
		    \mathbb{E}\left( ||a_{k}||_2^2 | \mathcal{F}_{k-1} \right) \leq \frac{Q_{m_k}}{M_k\bar{P}_k} \left( V + \left(V_G-1\right)||\nabla f(x^k(\theta),\theta)||_\pi^2\right).
		\end{equation*}
		We now turn our attention to the $b_{i,k}$ term. Here, we note that there is no dependence on $v$, and so find:
		\begin{align*}
		\mathbb{E}(b_{i,k}^2& | \mathcal{F}_{k-1}) \nonumber \\
		&= \frac{1}{M_k^2}\sum_{j=1}^{M_k}\mathbb{E}\left( \left( \nabla f(x^k(\theta_j), \theta_j) B_i(\theta_j) - \int_{\Theta} \nabla f(x^k(\theta),\theta) B_i(\theta) \pi(d\theta)\right)^2 | \mathcal{F}_{k-1} \right)
		\end{align*}
		then, on expanding, we find:
		\begin{equation*}
		\mathbb{E}(b_{i,k}^2 | \mathcal{F}_{k-1}) \leq \frac{1}{M_k} \int_\theta \nabla f(x^k(\theta),\theta)^2 B_i^2(\theta) \pi(d\theta).
		\end{equation*}
		After this, we recall the definition of $Q_{m_k}$, and sum the $b_{i,k}$ expectations over $i$ to find:
		\begin{equation*}
		\mathbb{E}(||b_k||_2^2 | \mathcal{F}_{k-1}) \leq \frac{Q_{m_k}}{M_k} \int \nabla f(x^k(\theta), \theta)^2 \pi(d\theta).
		\end{equation*}
		Combining the bounds on $||a_k||_2^2$ and $||b_k||_2^2$ via Equation \eqref{lem::ErrorEquationSum} gives:
		\begin{equation*}
		    \mathbb{E}(||e_k||_2^2 | \mathcal{F}_{k-1}) \leq \frac{Q_{m_k}}{M_k}\left(\hat{V}_k + \hat{V}_{G,k}||\nabla f(x^k(\theta),\theta)||_\pi^2\right)
		\end{equation*}
		with $\hat{V}_k = V/\bar{P}_k$ and $\hat{V}_{G,k} = 1 + (V_G-1)/\bar{P}_k$ as before, which gives our result.
\end{proof}

\begin{proof}[Proof of Theorem \ref{theorem::GDFixedUQ}]
		We first define the quantity $||u^{k+1} - u^*_m||_{2}^2 = d_{k+1}^2$, written in the $l^2$ norm. As stated earlier, $u^{k+1}$ is still an infinite-dimensional vector, just with zeros above the entry $m$, meaning we have:
		\begin{equation*}
			||u^k - u^*_m||_2^2 = \sum_{i=0}^{m-1}(u^k_i - (u^*_m)_i)^2
		\end{equation*}
		for all $k$. We write the direction vector as $D f(x)$, suppressing the basis level subscript and $\theta$ argument for clarity, and note this will have entries of zero above the truncation level. Our noisy estimate $g'_k$ will be written as $D' f(x^k)$. From the definition of $u^{k+1} = u^k - \gamma D' f(x^k)$, and since $d_k$ is $\mathcal{F}_{k-1}$-measurable, we then have:
		\begin{equation*}
		\mathbb{E}\left(d_{k+1}^2 | \mathcal{F}_{k-1}\right) =  d_k^2 - 2\gamma \mathbb{E}\left(\langle D' f(x^k), u^k - u^*_m \rangle_{2}| \mathcal{F}_{k-1}\right) + \gamma^2 \mathbb{E}\left(||D' f(x^k)||_{2}^2| \mathcal{F}_{k-1}\right),
		\end{equation*}
		where $x^k(\theta) = \mathcal{I}^{-1}(u^k)$. We look at the second term first.
        Taking the conditional expectation, $D' f(x^k)$ becomes $D f(x^k)$, while
        $u^k - u^*_m$ is $\mathcal{F}_{k-1}$-measurable. Since $u^k - u^*_m \in l^2_m$
        and the first $m$ entries of $D f(x^*_m)$ vanish,
        \begin{equation*}
        \langle D f(x^k), u^k - u^*_m \rangle_2
        = \langle \nabla f(x^k(\theta),\theta) - \nabla f(x^*_m(\theta),\theta),
        x^k(\theta) - x^*_m(\theta) \rangle_\pi .
        \end{equation*}
        Applying the strong convexity inequality on all of $L^2_\pi$ (the same
        inequality used in the proof of Lemma \ref{lemma::LinConvRemainder}, which is Theorem \ref{appendix::strongCocoercivity}) gives
        \begin{align*}
        \mathbb{E}\left(\langle D' f(x^k), u^k - u^*_m \rangle_2 | \mathcal{F}_{k-1}\right)
        &\geq \frac{\mu L}{\mu + L}||u^k - u^*_m||_2^2 \\
        &\quad + \frac{1}{\mu + L}
        ||\nabla f(x^k(\theta),\theta) - \nabla f(x^*_m(\theta),\theta)||_\pi^2 .
        \end{align*}
        We stress that the last term is the full $\pi$-norm of the gradient
        difference, not the norm of the truncated direction $D f(x^k)$; this is
        what allows it to be combined with the noise model below.
		By the noise model, the third term satisfies:
		\begin{equation*}
		\gamma^2 \mathbb{E}\left(||D' f(x^k)||_2^2 | \mathcal{F}_{k-1}\right) \leq  \gamma^2 C + \gamma^2 C_G||\nabla f(x^k(\theta), \theta) - \nabla f(x^*_m(\theta), \theta)||_{\pi}^2.
		\end{equation*}
		Substituting these expressions, we find:
		\begin{align*}
		\mathbb{E}(d_{k+1}^2 | \mathcal{F}_{k-1})& \\ 
		&\leq d_k^2 - 2\gamma\frac{\mu L}{\mu + L} d_k^2 + \gamma^2C \nonumber \\
		&+\gamma\left(\gamma C_G - \frac{2}{\mu + L}\right)||\nabla f(x^k(\theta), \theta) - \nabla f(x^*_m(\theta), \theta)||_{\pi}^2 \\
		&\leq \left(1 - \frac{2\gamma \mu L}{\mu + L}\right)d_k^2 + \gamma^2C
		\end{align*}
		where we impose the condition $\gamma - \frac{2}{C_G(\mu+L)} \leq 0$ in the final inequality. Taking total expectations and rearranging gives:
		\begin{equation*}
        \mathbb{E}(d_{k+1}^2) - \frac{(\mu+L)\gamma C}{2 \mu L} \leq \left(1 - \frac{2\gamma \mu L}{\mu + L}\right)\left(\mathbb{E}(d_k^2) - \frac{(\mu+L)\gamma C}{2 \mu L}\right)
		\end{equation*}
		Iterating this inequality down to $k = 1$, and using that $d_1$ is deterministic since $u^1$ is, gives the statement of the Theorem.
\end{proof}

\begin{proof}[Proof of Corollary \ref{cor:complexity}]
        We first note that our choice of step size implies that the condition on $M_k$ in Theorem \ref{theorem::gradConvLinear} states $M_k \geq 2 \hat V_{G,k} Q_{m_k}$, which we satisfy as $\hat V_{G,k} \leq V_G$. Next, let $r = \rho^{1/4}$. Equation \eqref{eq:regularity}, and the hypothesis on $m_k$ in Equation \eqref{eq:geometric-schedule} gives us:
        \begin{equation*}
            \|R_{m_k}\|_2 \le A m_k^{-p} \le A r^k.
        \end{equation*}
        Next, note that as $V = 0$, $\sigma_k = 0$. Finally, as $\rho < 1$, we have $r > \rho^{1/2}$. All of this combined means we can apply Corollary \ref{corol::exactLinear} with $K = 1, B = 0$. Doing so we see:
        \begin{equation*}
            \mathbb{E}\|u^{N+1} - u^*\|_2
            \leq \Big( \|u^1 - u^*\|_2 + \frac{c_R A\, r}{r - \rho^{1/2}} \Big) r^N .            
        \end{equation*}
        We separately bound the terms $r / (r - \rho^{1/2})$, and $c_R$. Firstly, we will bound $r / (r - \rho^{1/2}) = 1/(1 - \rho^{1/4})$ by seeing:
        \begin{equation*}
            1 - \rho^{1/4} \geq \frac{(1- \rho)}{4} = \frac{\mu L}{2(\mu + L)^2} \geq \frac{1}{8\kappa},
        \end{equation*}
        by using $(\mu + L)^2 \leq 4L^2$. We can also bound $c_R$ by doing:
        \begin{equation*}
            c_R = (1 + \rho^{1/2})\sqrt{\kappa + 1} + \frac{\sqrt{2} L }{\mu + L} \leq 2 \sqrt{\kappa + 1} + \sqrt{2},
        \end{equation*}
        and the bound on the iterate error comes from combining these bounds to form $D$.

        To show the required result on the number of gradient evaluations, we first seek to bound the number of iterations required for convergence. First, take the standard result for the number of iterations required to reach target accuracy $\delta$:
        \begin{equation*}
            N_{\delta} = \lceil \log(D/\delta) / \log(1/r) \rceil,
        \end{equation*}
        which implies $Dr^{N_\delta} \leq \delta$. We note that $\log(1/r) = \frac14 \log(1/\rho) \geq \frac14(1-\rho) \geq 1/(8\kappa)$ by using our earlier bounds. Combining this, we see $N_\delta \leq 8\kappa \log(D/\delta) + 1$.

        Finally, in iteration $k$, we can determine our cost with $M_k$. We know $M_k \leq 2 V_G Q_{m_k} + 1 \leq 2V_G B m^q_k + 1$ evaluations. We can also bound $m_k \leq r^{-k/p}+1 \leq 2r^{-k/p}$. Now, letting $s = r^{-q/p} > 1$, we can say:
        \begin{equation*}
            W(\delta) = \sum_{k=1}^{N_\delta} M_k
    \le N_\delta + 2^{q+1} V_G B \sum_{k=1}^{N_\delta} s^k
    \le N_\delta + 2^{q+1} V_G B\, \frac{s^{N_\delta + 1}}{s - 1}.
        \end{equation*}
        We can bound the denominator of the geometric sum by using $e^x-1 \geq x$ to show $s - 1 \ge \tfrac{q}{p} \log(1/r) \ge \tfrac{q}{8 p \kappa}$. The numerator can be bounded as $s^{N_\delta + 1} \le s^2 (D/\delta)^{q/p}$, because $N_\delta \le \log(D/\delta)/\log(1/r) + 1$ and $s^{\log(D/\delta)/\log(1/r)} = (D/\delta)^{q/p}$. For our final bound, we have $s^2 = \rho^{-q/(2p)} \le 2^{q/(2p)}$ because $\rho \ge 1/2$.  Combining all of these, we get
        \begin{equation*}
                W(\delta) \le N_\delta + 2^{q+1} V_G B \cdot 2^{q/(2p)} \cdot \frac{8 p \kappa}{q}
        \left(\frac{D}{\delta}\right)^{q/p}.
        \end{equation*}
        The result in Equation \eqref{eq:complexity} follows from using the bound on $N_\delta$.
\end{proof}

\section{Proof of Theorem \ref{theorem::AccGradLevel}} \label{sec:suppAGDFixed}
In this section we prove Theorem \ref{theorem::AccGradLevel}, the fixed-level result for the accelerated method.

\subsection{The method as a dynamical system} \label{sec:suppAGDDynamics}
Following \cite{aybat2019universally}, we write Algorithm \ref{algo_accel} at a fixed level $m$ as the dynamical system
\begin{equation*}
\xi^{k+1} = A\xi^k + B(g_k + e_k), \qquad y^k = C'\xi^k,
\end{equation*}
where $\xi^k = [(u^k)^T \quad (u^{k-1})^T]^T$ is the state vector, $g_k = D_m(f(z^k(\theta),\theta))$ is the noiseless gradient at $y^k$, $e_k$ is the gradient noise of Equation \eqref{equation::errVarianceAccel}, and $A = \tilde{A} \otimes I$, $B = \tilde{B} \otimes I$ and $C' = \tilde{C} \otimes I$ are the Kronecker products of
\begin{equation*}
\tilde{A} =
\begin{pmatrix}
1+\beta & -\beta \\
1 & 0
\end{pmatrix},
\qquad
\tilde{B} =
\begin{pmatrix}
-\alpha \\
0
\end{pmatrix},
\qquad
\tilde{C} =
\begin{pmatrix}
1+\beta & -\beta
\end{pmatrix}
\end{equation*}
with the identity of appropriate size. In the notation of Theorem \ref{theorem::AccGradLevel}, $\tilde{C}^T = c$, and we write $\tilde{P}_\alpha = p_\alpha p_\alpha^T$ and $P_\alpha = \tilde{P}_\alpha \otimes I$, so that $\tilde{Q}_\alpha = \tilde{P}_\alpha + 2\alpha C_G \tilde{C}^T\tilde{C}$ and $(\tilde{P}_\alpha)_{11} = \frac{1}{2\alpha}$.

\subsection{A one-step Lyapunov inequality}
The main step is the following Lemma, which extends \cite[Lemma 2.2]{aybat2019universally} to the noise model \eqref{equation::errVarianceAccel}, as is done inside the proof of \cite[Theorem K.1]{aybat2019universally}.
\begin{lemma} \label{lem:suppAccProof}
	Under all our usual hypotheses, and that there exists $\rho \in (0,1)$ and $\tilde{P} \in \mathbb{S}_+^2$, possibly depending on $\rho$ such that:
	\begin{equation*}
	\rho^2\tilde{X_1} + (1-\rho^2)\tilde{X_2} \succeq \begin{pmatrix}
	\tilde{A}^T\tilde{P}\tilde{A} - \rho^2 \tilde{P} & \tilde{A}^T\tilde{P}\tilde{B} \\
	\tilde{B}^T\tilde{P}\tilde{A} & \tilde{B}^T\tilde{P}\tilde{B}
	\end{pmatrix}
	\end{equation*}
	where:
	\begin{equation*}
	\tilde{X_1} = 
	\frac12 \begin{pmatrix}
	\beta^2 \mu & -\beta^2 \mu & -\beta \\
	-\beta^2 \mu & \beta^2 \mu & \beta \\
	-\beta & \beta & \alpha(2 - L\alpha)
	\end{pmatrix}
	\end{equation*}
	and:
	\begin{equation*}
	\tilde{X_2} = 
	\frac12 \begin{pmatrix}
	(1+\beta)^2 \mu & -\beta(1+\beta) \mu & -(1+\beta) \\
	-\beta(1+\beta) \mu & \beta^2 \mu & \beta \\
	-(1+\beta) & \beta & \alpha(2 - L\alpha)
	\end{pmatrix}.
	\end{equation*}
	Then let $P = \tilde{P} \otimes I$. We have for all $k \geq 0$:
	\begin{equation*}
	\mathbb{E}(V_P(\xi^{k+1})| \mathcal{F}_{k-1}) \leq \rho^2 V_P(\xi^{k}) + \alpha^2\left(C + C_G||y^k - u^*_m||_2^2 \right)\left(\frac{L}{2} + \tilde{P}_{11}\right)
	\end{equation*}
\end{lemma}

Throughout this section we work at a fixed level $m$, as in Theorem \ref{theorem::AccGradLevel}. We write $x^*_m$ and $u^*_m$ for the level $m$ optimum, $\xi^* = [(u^*_m)^T \quad (u^*_m)^T]^T$, and $g_k = D_m(f(z^k(\theta),\theta))$ for the noiseless gradient at $y^k$, so that $u^{k+1} = y^k - \alpha(g_k + e_k)$. The quantities $\xi^k$, $y^k$ and $g_k$ are $\mathcal{F}_{k-1}$-measurable, and $h(\xi) = \bar{f}(\mathcal{I}^{-1}(T\xi)) - \bar{f}(x^*_m)$ as in the main text. We use two inequalities for $\bar{f}$ on $L^2_{\pi,m}$: the upper bound of Lemma \ref{appendix::descent}, and the lower bound of Lemma \ref{appendix::lowerBound}, which we refer to as the strong convexity inequality \textit{in expectation}. As noted at the start of the previous section, both hold on $L^2_{\pi,m}$ with the gradient $D_m f$ in place of $\nabla f$.

Before we begin the proof, we need two preliminary Lemmas, adapted from \cite{aybat2019universally} and \cite{aybat2020robust}.

\begin{lemma} \label{appendix::WLemma}
Let $W_P(\xi) = (\xi - \xi^*)^T P(\xi-\xi^*)$, where $P = \tilde{P} \otimes I$ and $\tilde{P} \in \mathbb{S}_+^2$. Then
\begin{equation*}
\mathbb{E}(W_P(\xi^{k+1}) \mid \mathcal{F}_{k-1}) = \begin{pmatrix} \xi^k - \xi^* \\ g_k \end{pmatrix}^T \begin{pmatrix} A^TPA & A^TPB \\ B^TPA & B^TPB \end{pmatrix} \begin{pmatrix} \xi^k - \xi^* \\ g_k \end{pmatrix} + \alpha^2\tilde{P}_{11}\,\mathbb{E}(||e_k||_2^2 \mid \mathcal{F}_{k-1}),
\end{equation*}
and the last term is at most $\alpha^2\tilde{P}_{11}\left(C + C_G||y^k - u^*_m||_2^2\right)$.
\end{lemma}
\begin{proof}
Since $D_m(f(x^*_m(\theta),\theta)) = 0$ and $\tilde{A}(1,1)^T = (1,1)^T$, the point $\xi^*$ is a fixed point of the noiseless dynamics, so $\xi^{k+1} - \xi^* = A(\xi^k - \xi^*) + B g_k + B e_k$. Expanding the quadratic form, the cross terms involving $e_k$ vanish in conditional expectation since $\mathbb{E}(e_k \mid \mathcal{F}_{k-1}) = 0$, and $e_k^T B^TPB e_k = \alpha^2\tilde{P}_{11}||e_k||_2^2$ since $\tilde{B} = (-\alpha, 0)^T$. The final bound is the noise model \eqref{equation::errVarianceAccel}.
\end{proof}

\begin{lemma} \label{appendix::hLemma}
Let $X_i = \tilde{X}_i \otimes I$ for $i = 1, 2$. Then for any $\rho \in (0,1)$,
\begin{equation*}
\begin{aligned}
&\begin{pmatrix} \xi^k - \xi^* \\ g_k \end{pmatrix}^T \left(\rho^2 X_1 + (1-\rho^2)X_2\right) \begin{pmatrix} \xi^k - \xi^* \\ g_k \end{pmatrix} \\
&\qquad \leq \rho^2 h(\xi^k) - \mathbb{E}(h(\xi^{k+1}) \mid \mathcal{F}_{k-1}) + \frac{L\alpha^2}{2}\left(C + C_G||y^k - u^*_m||_2^2\right).
\end{aligned}
\end{equation*}
\end{lemma}
\begin{proof}
Identical to Lemma 4.5 of \cite{aybat2020robust}, except that in place of the inequalities for strongly convex functions with Lipschitz gradient used in their proof, we use Lemma \ref{appendix::descent} and the strong convexity inequality in expectation; their proof is pathwise, and the cross term $\alpha(1 - L\alpha)\,g_k^T e_k$ it produces vanishes in conditional expectation, and that the variance of the noise is bounded by $C + C_G||y^k - u^*_m||_2^2$ in place of a constant.
\end{proof}

\begin{proof}[Proof of Lemma \ref{lem:suppAccProof}]
Write $z_k = [(\xi^k - \xi^*)^T \quad g_k^T]^T$ and $N_k = C + C_G||y^k - u^*_m||_2^2$. Since $W_P(\xi^k) = z_k^T\,\mathrm{diag}(P, 0)\,z_k$, Lemma \ref{appendix::WLemma} and the matrix inequality in the hypothesis, which carries over to the Kronecker products with $I$, give
\begin{equation*}
\mathbb{E}(W_P(\xi^{k+1}) \mid \mathcal{F}_{k-1}) - \rho^2 W_P(\xi^k) \leq z_k^T\left(\rho^2 X_1 + (1-\rho^2)X_2\right)z_k + \alpha^2\tilde{P}_{11}N_k .
\end{equation*}
By Lemma \ref{appendix::hLemma}, the right hand side is at most $\rho^2 h(\xi^k) - \mathbb{E}(h(\xi^{k+1}) \mid \mathcal{F}_{k-1}) + \left(\frac{L\alpha^2}{2} + \alpha^2\tilde{P}_{11}\right)N_k$. Since $V_P = W_P + h$, rearranging gives
\begin{equation*}
\mathbb{E}(V_P(\xi^{k+1}) \mid \mathcal{F}_{k-1}) \leq \rho^2 V_P(\xi^k) + \alpha^2\left(\frac{L}{2} + \tilde{P}_{11}\right)N_k ,
\end{equation*}
which is the statement of the Lemma.
\end{proof}

\subsection{Proof of the Theorem}
\begin{proof}[Proof of Theorem \ref{theorem::AccGradLevel}]
	Our proof follows that of \cite[Theorem K.1]{aybat2019universally}, whose noise model is Equation \eqref{equation::errVarianceAccel} with $\sigma^2 = C$ and $\eta^2 = C_G$. Each step of that proof is a one-step inequality, and so holds conditionally on $\mathcal{F}_{k-1}$. By \cite[Theorem 2.3]{aybat2019universally}, the matrix inequality of Lemma \ref{lem:suppAccProof} holds with $\rho^2 = 1 - \sqrt{\alpha\mu}$ and $\tilde{P} = \tilde{P}_\alpha$. Applying Lemma \ref{lem:suppAccProof}, and using $\alpha^2\left(\frac{L}{2} + (\tilde{P}_\alpha)_{11}\right) = \alpha^2\left(\frac{L}{2} + \frac{1}{2\alpha}\right) \leq \alpha$ since $\alpha L \leq 1$, we find:
	\begin{equation*}
	\mathbb{E}(V_{P_\alpha}(\xi^{k+1})| \mathcal{F}_{k-1}) \leq \left(1 - \sqrt{\alpha \mu}\right)V_{P_\alpha}(\xi^{k}) + \alpha\left(C + C_G||y^k - u^*_m||_2^2 \right)
	\end{equation*}
	By construction we have $y^k = C'\xi^k$ so we can write $||y^k - u^*_m||_2^2 = (\xi^k-\xi^*)^T C'^T C' (\xi^k - \xi^*)$. Using that, we can rewrite the previous Equation as:
	\begin{align*}
	\mathbb{E}(V_{P_\alpha}(\xi^{k+1})| \mathcal{F}_{k-1}) &\leq \left(1 - \sqrt{\alpha \mu}\right)V_{P_\alpha}(\xi^{k}) \\
	&+ \frac12 (\xi^k-\xi^*)^T 2\alpha C_G C'^T C' (\xi^k - \xi^*) + \alpha C \\
	&\leq \left(1 - \sqrt{\alpha \mu}\right)V_{Q_\alpha}(\xi^{k}) + \alpha C
	\end{align*}
	where the final inequality follows from $1 - \sqrt{\alpha \mu} \geq \frac12$, which holds since $\alpha \leq 1/L$ and $\kappa \geq 4$. It remains to replace $V_{P_\alpha}(\xi^{k+1})$ on the left hand side by $V_{Q_\alpha}(\xi^{k+1}) = V_{P_\alpha}(\xi^{k+1}) + 2\alpha C_G(\xi^{k+1}-\xi^*)^T C'^T C' (\xi^{k+1} - \xi^*)$. Using $(a+b)^2 \leq 2a^2 + 2b^2$, $\beta \leq 1$, and $h(\xi) \geq \frac{\mu}{2}||T\xi - u^*_m||_2^2$, together with the fact that both Lyapunov functions are at least $h$, we have
    \begin{align*}
    (\xi^{k+1}-\xi^*)^T C'^T C' (\xi^{k+1} - \xi^*)
    &=
    \left\|(1+\beta)(u^{k+1} - u^*_m)
    - \beta(u^k - u^*_m)\right\|_2^2 \\
    &\leq
    8\left\|u^{k+1} - u^*_m\right\|_2^2
    + 2\left\|u^k - u^*_m\right\|_2^2 \\
    &\leq
    \frac{16}{\mu}V_{P_\alpha}(\xi^{k+1})
    + \frac{4}{\mu}V_{Q_\alpha}(\xi^k).
    \end{align*}
	Hence $V_{Q_\alpha}(\xi^{k+1}) \leq \left(1 + \frac{32\alpha C_G}{\mu}\right)V_{P_\alpha}(\xi^{k+1}) + \frac{8\alpha C_G}{\mu}V_{Q_\alpha}(\xi^k)$, and taking conditional expectations and combining with the previous display gives
	\begin{equation*}
	\mathbb{E}(V_{Q_\alpha}(\xi^{k+1}) | \mathcal{F}_{k-1}) \leq \left(1 - \sqrt{\alpha\mu} + \frac{40\alpha C_G}{\mu}\right)V_{Q_\alpha}(\xi^k) + \left(1 + \frac{32\alpha C_G}{\mu}\right)\alpha C .
	\end{equation*}
	Finally, $\alpha \leq \mu^3/(60C_G)^2$ gives $60\alpha C_G \leq \sqrt{\alpha}\mu^{3/2}$, so that $\frac{40\alpha C_G}{\mu} \leq \frac{2}{3}\sqrt{\alpha\mu}$ and $\frac{32\alpha C_G}{\mu} \leq \sqrt{\alpha\mu} \leq 1$, which gives the result.
\end{proof}

\subsection{Proof of Corollary \ref{corol:UQAGDFixed}} \label{sec:suppAGDCorollary}
\begin{proof}[Proof of Corollary \ref{corol:UQAGDFixed}]
    By strong convexity and the first order optimality of $u^*_m$ on $l^2_m$, $V_{Q_\alpha}(\xi) \geq h(\xi) \geq \frac{\mu}{2}||T\xi - u^*_m||_2^2$. Since $u^0 = u^1$, we also have $V_{Q_\alpha}(\xi^1) \leq \left(\frac{\mu + L}{2} + 2\alpha C_G\right)||u^1 - u^*_m||_2^2$. The result follows by iterating Theorem \ref{theorem::AccGradLevel}, taking total expectations, and summing the geometric series.
\end{proof}

The constant $1 + \kappa + 4\alpha C_G/\mu$ in front of $||u^1 - u^*_m||_2^2$ can be large when $\kappa$ is large. When $C_G = 0$ it equals $1 + \kappa$, and by the same reasoning as in \cite{van2017fastest}, it adds an extra $O(\sqrt{\kappa}\ln\kappa)$ iterations to convergence. As we are interested in high accuracy convergence, this term is small compared with the total number of iterations, so despite the unfavourable constant the method does attain acceleration.

\section{Short-sale probabilities in the portfolio problem} \label{sec:suppShortSale}
Both columns of Table \ref{tab:asset_statistics} are computed on the same $20,000$ samples of $\theta$, with the learned weights obtained by evaluating the chaos expansion at each sample. The differences between them therefore have much smaller Monte Carlo error than either column: for $P(w_i < 0)$ the standard error of the difference is below $0.001$ for every asset, against about $0.0035$ for each column as an estimate of the population probability.

We note that convergence in $L^2_\pi$ controls expectations, but not threshold probabilities in general. For example, if $w^* \equiv 0$ and $\hat w \equiv -1/k$, then $\hat w \to w^*$ in $L^2_\pi$ while $\Pr(\hat w < 0) = 1 \neq 0 = \Pr(w^* < 0)$. For a scalar component and any $t > 0$, however:
    \begin{equation*}
        \left|\Pr(\hat{w}_i < 0) - \Pr(w^*_i < 0)\right| \leq \Pr(|w^*_i| \leq t) + \frac{\|\hat{w}_i - w^*_i\|_{L^2_\pi}^2}{t^2},
    \end{equation*}
    since if $\hat{w}_i$ and $w^*_i$ differ in sign, then either $|w^*_i| \leq t$ or $|\hat{w}_i - w^*_i| > t$, and the second event is bounded by Chebyshev's inequality. If $w^*_i$ has a density bounded by $c$ near $0$, then $\Pr(|w^*_i| \leq t) \leq 2ct$, and taking $t = (\|\hat{w}_i - w^*_i\|_{L^2_\pi}^2/c)^{1/3}$ gives
    \begin{equation*}
        \left|\Pr(\hat{w}_i < 0) - \Pr(w^*_i < 0)\right| \leq 3\left(c\,\|\hat{w}_i - w^*_i\|_{L^2_\pi}\right)^{2/3}.
    \end{equation*}

\section{Experimental details} \label{sec:suppExpDetails}
\subsection{General setup}
All experiments were run on a MacBook M4 Max with a 14-core CPU. Unless stated otherwise, all random numbers are generated from a single base seed ($0$), with a distinct offset for each method, and the portfolio instance of Section \ref{ss:portfolioOpt} is generated with seed $0$. Reference solutions are computed as follows: for the trigonometric basis, by a $2^{14}$-point trapezoidal rule; for the portfolio problem, by projecting $x^*$ onto the Legendre polynomials of total degree at most $20$ ($231$ functions) with a $60 \times 60$ tensor Gauss--Legendre rule; and for the logistic regression problem, by Newton's method to a tolerance of $10^{-13}$ at Gauss--Legendre nodes.

\subsection{Benchmark quadratic} \label{sec:suppExpQuadratic}
In Figure \ref{fig::NesterovAveraged}, the Stochastic Approximation method of \cite{crepey2020uncertainty} uses step sizes that begin at $1/100$, to prevent initial blow-up, and decay like $1/k$, which is in their range of comparison. It uses the same Monte Carlo estimator for the gradient as our methods. All curves in Figure \ref{fig:figures} are averaged over $200$ trials.

In Figure \ref{fig::UQJustification}, the full UQ method took $11.6$ seconds for $200$ trials, compared to $19$ seconds for the fixed level algorithm. Allotting the same computational time to each, the UQ method reaches an error smaller by a factor of about $60$.

\subsection{Portfolio problem} \label{sec:suppExpPortfolio}
We set $\lambda = 2, \tau = 0.2, \eta=0.1, \gamma=3, s = (0.05, 0.025)$ and $w_0 = e/n$. After fixing a random seed, we draw returns $r_i \sim U[0,0.1]$, idiosyncratic variances $D_i \sim U[0.02, 0.1]$, market loadings $b_{1i} \sim U[0.5, 1.5]$ and sector loadings $b_{2i} \sim N(0,1)$. The budget constraint is imposed only through the penalty, so the optimal portfolio is not fully invested: the mean weights in Table \ref{tab:asset_statistics} sum to $0.83$. The remainder can be read as a holding in cash, with zero return and zero variance, and the penalty as a cost on deviating from full investment. The purpose of this example is to illustrate the UQ output, so we do not tune the penalty further.

In the comparison with other methods, LSP and ALSP use the step size $1/(Q_m L)$ of \cite{dong2025beyond}, and ALSP starts from total degree 1 ($m = 3$) and adds one function at a time. The naive method solves each sample by gradient descent with step $2/(\mu+L)$ to a squared relative accuracy of $10^{-8}$ (about 384 iterations).

In every method that uses the stall rule, including ALSP, the objective at a check is computed by Gauss--Legendre quadrature ($64 \times 64$ nodes for the portfolio problem, $128$ nodes for the logistic regression problem), so the check is deterministic. As in \cite{dong2025beyond}, these objective evaluations are not counted as gradient evaluations, for any method.

\subsection{Logistic regression problem} \label{sec:suppExpLogistic}
As the loss is logistic, we have $\ell_i'' \le 1/4$. Then, writing $\tilde a_i = (a_i, 1)$ and $\Sigma_\pm = \frac{1}{n_\pm}\sum_{i:\, y_i = \pm} \tilde a_i \tilde a_i^\top$ for the second-moment matrices of the two classes, we have $\nabla^2 f(x,\lambda) \preceq \tfrac{1-\lambda}{4}\Sigma_+ + \tfrac{\lambda}{4}\Sigma_- + \rho I$. As the bound is affine in $\lambda$, its largest eigenvalue over $[0,1]$ must be attained at an endpoint, and $L = \rho + c$ with $c = \tfrac14 \max\{\lambda_{\max}(\Sigma_+), \lambda_{\max}(\Sigma_-)\} \approx 3.67$ (the negative class attains the maximum).

For each $\rho$ we take the smallest $m_{\max}$ with $\|R_{m_{\max}}\|_2^2 \le 10^{-8}\|u^*\|_2^2$, which gives $m_{\max} = 12$, $16$, $25$ and $39$. The reference coefficients are computed by Newton's method at Gauss--Legendre nodes. They agree with the path published by \cite{dong2025beyond} to within $2.5\times10^{-7}$. ALSP uses the step size $1/(Q_m L)$ of \cite{dong2025beyond}.

For each $\kappa$, ALSP's check interval (250, 1000 or 4000 steps) and the accelerated method's check interval (2 or 5 iterations) were tuned with the same selection rule, namely the minimal number of gradient evaluations needed to reach an accuracy of $10^{-5}$. At $\kappa \approx 940$, where each ALSP run is long, we used 5 trials and a budget of about $8.6 \times 10^6$ steps ($6\,Q_{m_{\max}}\kappa$, against $25\,Q_{m_{\max}}\kappa$ for the other values of $\kappa$), well above the $3.3\times10^6$ steps ALSP needs to reach $10^{-7}$. There we tuned ALSP over intervals of 4000, 16000 and 64000 steps (16000 was best), and kept the accelerated method's interval from $\kappa \approx 236$.

Table \ref{tab:logistic} gives the speedups plotted in Figure \ref{fig:logistic}(b), together with those at a relative error of $10^{-3}$.

\begin{table}[htbp]
      \centering
      \caption{Speedup of the accelerated method over the tuned ALSP on the logistic
        regression \eqref{eq:logistic}, in gradient evaluations to first reach each relative
        error. Both methods use ALSP's stall rule. Means over 20 trials, except 5 at
        $\kappa \approx 940$ (see above).}
      \label{tab:logistic}
      \footnotesize
      \begin{tabular}{cccccc}
        \toprule
        $\rho$ & $\kappa$ & $\sqrt{\kappa}/2$ & $10^{-3}$ & $10^{-5}$ & $10^{-7}$ \\
        \midrule
        $1/4$ (\cite{dong2025beyond}) & 16 & 2.0 & 4.2 & 2.6 & 1.9 \\
        $1/16$ & 60 & 3.9 & 7.7 & 4.1 & 2.9 \\
        $1/64$ & 236 & 7.7 & 14.1 & 6.1 & 4.4 \\
        $1/256$ & 940 & 15.3 & 23.4 & 11.6 & 11.3 \\
        \bottomrule
      \end{tabular}
    \end{table}

\end{document}